\documentclass[reqno]{amsart}

\usepackage{hyperref}

\newcommand*{\mailto}[1]{\href{mailto:#1}{\nolinkurl{#1}}}

\newtheorem{theorem}{Theorem}[section]
\newtheorem{definition}[theorem]{Definition}
\newtheorem{lemma}[theorem]{Lemma}
\newtheorem{proposition}[theorem]{Proposition}
\newtheorem{corollary}[theorem]{Corollary}

\newtheorem{remark}[theorem]{Remark}

\newcommand{\R}{{\mathbb R}}
\newcommand{\N}{{\mathbb N}}

\newcommand{\C}{{\mathbb C}}

\newcommand{\spr}[2]{\langle #1 , #2 \rangle}

\newcommand{\E}{\mathrm{e}}
\newcommand{\I}{\mathrm{i}}

\newcommand{\loc}{{\mathrm{loc}}}
\newcommand{\cc}{{\mathrm{c}}}

\newcommand{\res}[1]{\mathrm{res}_{#1}\,}

\newcommand{\OO}{\mathcal{O}}

\newcommand{\cprod}{\prod^{\leftrightarrow}}
\newcommand{\csum}{\sum^{\leftrightarrow}}

\newcommand{\ledot}{\,\cdot\,}
\newcommand{\redot}{\cdot\,}

\newcommand{\Iso}[1]{\mathrm{Iso}(#1)}

\newcommand{\qd}{{[1]}}
\newcommand{\qpm}{{[\pm]}}
\newcommand{\qp}{{[+]}}
\newcommand{\qm}{{[-]}}
\newcommand{\dip}{\upsilon}
\newcommand{\D}{\mathcal{D}}
\newcommand{\De}{\mathcal{E}}
\newcommand{\Ar}{\mathsf{a}}

\renewcommand{\labelenumi}{(\roman{enumi})}
\renewcommand{\theenumi}{\roman{enumi}}
\numberwithin{equation}{section}

\begin{document}

\title[The inverse spectral transform]{The inverse spectral transform for the conservative Camassa--Holm flow with decaying initial data}
 
\author[J.\ Eckhardt]{Jonathan Eckhardt}
\address{Faculty of Mathematics\\ University of Vienna\\ Oskar-Morgenstern-Platz 1\\ 1090 Wien\\ Austria}
\email{\mailto{jonathan.eckhardt@univie.ac.at}}
\urladdr{\url{http://homepage.univie.ac.at/jonathan.eckhardt/}}

\thanks{\href{http://dx.doi.org/10.1007/s00205-016-1066-z}{Arch.\ Ration.\ Mech.\ Anal.\ {\bf 224} (2017), no.~1, 21--52}}
\thanks{{\it Research supported by the Austrian Science Fund (FWF) under Grant No.\ J3455}}


\begin{abstract}
 We establish the inverse spectral transform for the conservative Camassa--Holm flow with decaying initial data.
 In particular, it is employed to prove existence of weak solutions for the corresponding Cauchy problem. 
\end{abstract}

\maketitle

\section{Introduction}

{\sf T}he {\em inverse scattering transform} \cite{abcl91, abse81, ecvH81} is a powerful tool for solving initial value problems for certain nonlinear partial differential equations. 
 Originally developed for the Korteweg--de Vries equation by Gardner, Greene, Kruskal and Miura \cite{gagrkrmi67}, this method has since been successfully extended to various other completely integrable equations.
 It is the aim of the present article to establish the corresponding transform for the Camassa--Holm equation
  \begin{equation}\label{eqnCH}
   u_{t} -u_{xxt}  = 2u_x u_{xx} - 3uu_x + u u_{xxx} 
  \end{equation}
  with decaying initial data. 
  Due to the vast amount of literature on this equation, we only refer to a selection of articles  \cite{besasz00, bosh06, brco07, co01, coes98, como00, cost00, geho08, hora07, le05b, mc04, mc03, xizh00} containing further information. 
  The relevance of the Camassa--Holm equation stems from the fact that it constitutes a model for unidirectional wave propagation on shallow water \cite{caho93, cola09, io07, jo02}. 
  Unlike the Korteweg--de Vries equation, it allows for smooth solutions to blow up in finite time in a way that resembles wave-breaking. 
 This process has been described in detail \cite{coes98, coes98b, mc04, mc03} and is known to only happen when the quantity $\omega=u-u_{xx}$ changes sign. 
 Compared to the rather tame sign-definite case (which shares a lot of similarities with the Korteweg--de Vries equation \cite{besasz98, le04, mc03b}), indefiniteness of $\omega$ causes serious complications (for example noticeable from the discussions in \cite{comc99, ka06, le05}). 
  Many of these problems are already apparent for the prototypical example of a peakon-antipeakon collision \cite{besasz01, coes98, wa06}, which is a special case of so-called multi-peakon solutions \cite{besasz00}. 

{\sf O}n the other side, the Camassa--Holm equation is known to be formally completely integrable in the sense that there is an associated isospectral problem  
\begin{align}\label{eqnISPcla}
 -f'' + \frac{1}{4} f = z\, \omega f, 
\end{align}
where $z$ is a complex spectral parameter.  
Solving the corresponding inverse problem is essentially equivalent to solving the initial value problem for the Camassa--Holm equation. 
For this reason, it is not surprising that the encountered complications due to wave-breaking for indefinite $\omega$ reoccur when dealing with this inverse problem. 
 And in fact, despite a large amount of articles, very little is known about the inverse problem for~\eqref{eqnISPcla} in the indefinite case and almost the entire literature on this subject restricts to strictly positive and smooth $\omega$ (in which case the spectral problem can be transformed into a standard form that is known from the Korteweg--de Vries equation  \cite{besasz98, le04, mc03b}).  
Apart from the explicitly solvable finite dimensional case \cite{besasz00, ConservMP}, only insufficient partial uniqueness results \cite{be04, bebrwe08, bebrwe12, bebrwe15, LeftDefiniteSL, CHPencil, IsospecCH} have been obtained so far for the inverse problem in the indefinite case.  

{\sf M}aking use of recent progress on the inverse spectral problem for indefinite strings in \cite{IndefiniteString}, we will be able to overcome these difficulties and establish the inverse spectral transform for the more general two-component Camassa--Holm system
 \begin{align}
 \begin{split}\label{eqnmy2CH}
  u_t + u u_x + P_x & = 0, \\
  \mu_t + (u\mu)_x & = (u^3 - 2Pu)_x, 
 \end{split}
 \end{align}
 where the auxiliary function $P$ satisfies
 \begin{align}
  P - P_{xx} & = \frac{u^2+ \mu}{2},
 \end{align} 
 for a class $\De$ of decaying initial data to be defined below; this should be compared to the definition of the set $\D$ in \cite[Section~6]{brco07}, \cite[Definition~3.1]{hora07}, \cite[Definition~4.1]{grhora12}.  
 Note that the two-component Camassa--Holm system~\eqref{eqnmy2CH} is % by means of the change of variables given by $\mu=u^2+u_x^2+\rho^2$, 
 equivalent to its more commonly used incarnation in \cite{chlizh06, coiv08, esleyi07, grhora12, hoiv11} (cf.\ \cite[Equation~(4.2)]{hora07}), which is known to contain the Camassa--Holm equation as a special case.

\begin{definition}
The set $\D$ consists of all pairs $(u,\mu)$ such that $u$ is a real-valued function in $H^1(\R)$ and $\mu$ is a non-negative finite Borel measure on $\R$ with
\begin{align}\label{eqnmuac}
 \mu(B) \geq \int_B u(x)^2 + u'(x)^2\, dx 
\end{align}
for every Borel set $B\subseteq\R$.
\end{definition}

{\sf A}ssociated with each pair $(u,\mu)\in\D$ is a distribution $\omega$ in $H^{-1}(\R)$ defined by 
\begin{align}\label{eqnDefomega}
 \omega(h) = \int_\R u(x)h(x)dx + \int_\R u'(x)h'(x)dx, \quad h\in H^1(\R),
\end{align}
so that $\omega = u - u''$ in a distributional sense,  as well as a non-negative finite Borel measure $\dip$ on $\R$ defined such that   
\begin{align}\label{eqnDefdip}
  \mu(B) =  \int_B u(x)^2 + u'(x)^2\, dx  + \dip(B) 
\end{align}
for every Borel set $B\subseteq\R$.
 Let us point out that it is always possible to uniquely recover the pair $(u,\mu)$ from the distribution $\omega$ and the Borel measure $\dip$.
 In fact, the  function $u$ at any point $x\in\R$ can be written as  
 \begin{align}
  u(x) & = \omega(\delta_x), & \delta_x(s) & = \frac{1}{2} \E^{-|x-s|}, \quad s\in\R,
 \end{align}
 which then allows us to determine the Borel measure $\mu$ as well from~\eqref{eqnDefdip}.   
 Now the class $\De$ is defined by imposing an additional growth restriction on pairs in $\D$. 

\begin{definition}
The set $\De_\pm$ consists of all pairs $(u,\mu)\in\D$ such that 
\begin{align}\label{eqnDedef}
 \int_\R \E^{\pm x} \left(u'(x)\mp u(x)\right)^2 dx +  \int_\R \E^{\pm x} d\dip(x) < \infty. 
\end{align}
Furthermore, the set $\De$ is defined as the intersection of $\De_+$ and $\De_-$. 
\end{definition}

{\sf J}ustified by the Lax pair formulation of the two-component Camassa--Holm system in \cite{chlizh06, coiv08, hoiv11} and the results in \cite{ConservMP}, we will consider the spectral problem  
\begin{align}\label{eqnISP}
 -f'' + \frac{1}{4} f = z\, \omega f + z^2 \dip f. 
\end{align}
Due to the low regularity of the coefficients, it is not clear how this differential equation has to be interpreted, which is why we clarify this matter in Appendix~\ref{appRelCan}. 
Basic properties of the spectral problem (like realness and discreteness of the spectrum $\sigma$ for example) will be discussed in Section~\ref{secSP} along with some further necessary conditions. 
In the following section, we will then solve the corresponding inverse problem for the class of coefficients corresponding to $\De$, giving a complete characterization of all possible spectral data. 
More precisely, we will establish a one-to-one correspondence between $\De$ and a class of spectral data explicitly described by Theorem~\ref{thmIP}, which will be shown to be a homeomorphism with respect to suitable topologies.
Finally, in Section~\ref{secCCH} we will introduce {\em the conservative Camassa--Holm flow} as a certain isospectral flow on $\De$ and show that its integral curves define weak solutions of the two-component Camassa--Holm system. 
 
{\sf A}s an immediate consequence of the solution of the inverse spectral problem in Theorem~\ref{thmIP}, we will see that the phase space $\De$ decomposes into a foliation of invariant isospectral sets $\Iso{\sigma}$, each of which can be parametrized by the set\footnote{We denote with $\R^\sigma$ the set of all real-valued sequences $\kappa = \{\kappa_\lambda\}_{\lambda\in\sigma}$ indexed by $\sigma$.}  
 \begin{align}
  \left\lbrace \kappa\in\R^\sigma \,\left|\; \sum_{\lambda\in\sigma} \frac{1}{\lambda^2} \frac{\E^{|\kappa_\lambda|}}{|\lambda \dot{W}(\lambda)|}  < \infty \right.\right\rbrace, 
 \end{align} 
 where $W$ is the entire function defined by the infinite product in~\eqref{eqncproddef} and the dot denotes differentiation.
 In terms of these coordinates, the conservative Camassa--Holm flow on each set $\Iso{\sigma}$ becomes the simple linear flow given by   
\begin{align}
 \kappa_\lambda' = \frac{1}{2\lambda}, \quad \lambda\in\sigma. 
\end{align}
These facts are reminiscent of the fact that the conservative Camassa--Holm flow can be viewed  as a completely integrable infinite dimensional Hamiltonian system.

\subsection*{Notation} 
  For integrals of a function $f$ which is locally integrable with respect to a Borel measure $\nu$ on an interval $I$, we will employ the convenient notation 
\begin{align}\label{eqnDefintmu}
 \int_x^y f\, d\nu = \begin{cases}
                                     \int_{[x,y)} f\, d\nu, & y>x, \\
                                     0,                                     & y=x, \\
                                     -\int_{[y,x)} f\, d\nu, & y< x, 
                                    \end{cases} \qquad x,\,y\in I, 
\end{align}
 rendering the integral left-continuous as a function of $y$. 
 If the function $f$ is locally absolutely continuous on $I$ and $g$ denotes a left-continuous distribution function of the Borel measure $\nu$, then we have the integration by parts formula %(see, for example, \cite[Exercise~5.8.112]{bo07}, \cite[Theorem~21.67]{hest65})
\begin{align}\label{eqnPI}
  \int_{x}^y  f\, d\nu = \left. g f\right|_x^y - \int_{x}^y g(s) f'(s) ds, \quad x,\,y\in I,
\end{align}
 which will be used frequently throughout this article. 

 Given a discrete set $\sigma$ of nonzero reals, we denote with $n_\sigma(r)$ the number of all $\lambda\in\sigma$ with modulus not greater than $r$.
 Furthermore, we introduce the notation
\begin{align}\label{eqncsumdef}
 \csum_{\lambda\in\sigma} \frac{1}{\lambda} = \lim_{r\rightarrow\infty} \mathop{\sum_{\lambda\in\sigma}}_{|\lambda|\leq r} \frac{1}{\lambda},
\end{align}
provided that the limit exists. 
Similarly, subject to existence, we shall write
\begin{align}\label{eqncproddef}
 \cprod_{\lambda\in\sigma} \biggl(1-\frac{z}{\lambda}\biggr) = \lim_{r\rightarrow\infty} \mathop{\prod_{\lambda\in\sigma}}_{|\lambda|\leq r} \biggl(1-\frac{z}{\lambda}\biggr), \quad z\in\C,
\end{align}
where the limit is meant to be taken in the topology of locally uniform convergence. 
The limit in~\eqref{eqncproddef} exists if and only if the limit in~\eqref{eqncsumdef} exists and the sum
\begin{align}
 \sum_{\lambda\in\sigma} \frac{1}{\lambda^2}
\end{align}
is finite. 
% => look at the derivatives
% <= add the exponentials to the product
%\begin{align}
%   \mathop{\prod_{\lambda\in\sigma}}_{|\lambda|\leq r} \biggl(1-\frac{z}{\lambda}\biggr) =  \E^{-z \mathop{\sum_{\lambda\in\sigma}}_{|\lambda|\leq r} \frac{1}{\lambda}}\mathop{\prod_{\lambda\in\sigma}}_{|\lambda|\leq r} \biggl(1-\frac{z}{\lambda}\biggr) \E^{\frac{z}{\lambda}}, \quad z\in\C,
%\end{align}
%      and note that both factors converge locally uniformly. 
In this case, upon denoting  the entire function in~\eqref{eqncproddef}  with $W$, we have 
\begin{align}\label{eqncprodatzero}
 \csum_{\lambda\in\sigma} \frac{1}{\lambda} & = - \dot{W}(0), & \sum_{\lambda\in\sigma} \frac{1}{\lambda^2} & = \dot{W}(0)^2 - \ddot{W}(0).
\end{align}

%%%%%%%%%%%%%%%%%%%%%%%
 \section{The direct spectral problem}\label{secSP}
%%%%%%%%%%%%%%%%%%%%%%%

In the present section, we are going to introduce the spectral quantities that will linearize the conservative Camassa--Holm flow on $\De$ and derive their basic properties. 
To this end, we fix an arbitrary pair $(u,\mu)\in\De$ and first recall the definition of the distribution $\omega$ in~\eqref{eqnDefomega} as well as the one of the Borel measure $\dip$ in~\eqref{eqnDefdip}. 
 We introduce the spectrum $\sigma$ associated with the pair $(u,\mu)$ as the set of all those numbers $z\in\C$ for which there is a nontrivial solution $f$ of the differential equation
\begin{align}\label{eqnDE}
 - f'' + \frac{1}{4} f = z\, \omega f + z^2 \dip f 
\end{align}
 that belongs to $H^1(\R)$. 
 Regarding the precise meaning and some basic properties of this differential equation we refer to the discussion in Appendix~\ref{appRelCan}.

 A first consequence of the growth restriction on $(u,\mu)$ in~\eqref{eqnDedef} is the existence of particular solutions of the differential equation~\eqref{eqnDE} with prescribed asymptotics.

\begin{theorem}\label{thmThetaPhi}
For every $z\in\C$ there is a unique solution $\phi_\pm(z,\redot)$ of the differential equation~\eqref{eqnDE} with the asymptotics 
\begin{align}\label{eqnphiasym}
  \phi_\pm(z,x) & \sim \E^{\mp\frac{x}{2}}, \qquad x\rightarrow\pm\infty.  
\end{align}
The derivative of $\phi_\pm(z,\redot)$ is integrable and square integrable near $\pm\infty$. 
\end{theorem}

\begin{proof}
 To begin with, let us introduce the diffeomorphism $\eta_\pm:\R_\mp\rightarrow\R$ by  
 %\marginpar{$\eta_\pm'(\xi) = \mp \frac{1}{\xi}$}
 %\marginpar{$\eta_\pm^{-1}(x) = \mp\E^{\mp x}$}
 %\marginpar{$\eta_\pm^{-1\prime}(x) = \E^{\mp x}$}
 \begin{align*}
  \eta_\pm(\xi) = \mp\ln(\mp\,\xi), \quad \xi\in\R_\mp,
 \end{align*} 
 where $\R_\mp$ denotes the open negative/positive semi-axis. %, explicitly that is $\R_- = (-\infty,0)$ and $\R_+=(0,\infty)$.
 Furthermore, let $\Ar_\pm$ be a measurable function on $\R_\mp$ such that % choose an everywhere defined representative for $u'$ and define $\Ar_\pm$ by the equation below  
 (we set $\alpha_\pm = - u' \pm u$ as in Appendix~\ref{appRelCan})
 \begin{align}\label{eqnDefa}
  \Ar_\pm(\xi) =  \frac{\alpha_\pm(\eta_\pm(\xi))}{\mp\xi} =  \frac{-u'(\eta_\pm(\xi))\pm u(\eta_\pm(\xi))}{\mp\xi}  
 \end{align}
 for almost all $\xi\in\R_\mp$ (note that the right-hand side is well-defined almost everywhere)
 and define the non-negative Borel measure $\beta_\pm$ on $\R_\mp$ via setting
 \begin{align}\label{eqnDefbeta}
  \beta_\pm(B) =  \int_B \frac{1}{\mp\xi}\, d\dip\circ\eta_\pm(\xi)  = \int_{\eta_\pm(B)} \E^{\pm x} d\dip(x)
 \end{align}
 for every Borel set $B\subseteq\R_\mp$. 
 The growth restriction on $(u,\mu)$ in~\eqref{eqnDedef} implies that the function $\Ar_\pm$ is square integrable and that  the measure $\beta_\pm$ is finite (thus $\beta_\pm$ can be extended to a Borel measure on the closure of $\R_\mp$ by setting $\beta_\pm(\lbrace0\rbrace) = 0$).
 Consequently (see, for example, \cite[Section~11.8]{at64}, \cite[Theorem~1.1]{be89}, \cite[Theorem~A.2]{MeasureSL}), there is a unique matrix solution $Y_\pm(z,\redot)$ on $\R_\mp$ of the integral equation 
 \begin{align}\begin{split}\label{eqnCanSysString}
  Y_\pm(z,\xi) = \begin{pmatrix} 1 & 0 \\ 0 & 1 \end{pmatrix} & + z \int_0^\xi \begin{pmatrix} -\Ar_\pm(s) & -1 \\ \Ar_\pm(s)^2 & \Ar_\pm(s) \end{pmatrix} Y_\pm(z,s)ds \\ 
       & + z \int_{0}^{\xi} \begin{pmatrix} 0 & 0 \\ 1 & 0 \end{pmatrix} Y_\pm(z,s) d\beta_\pm(s), \quad \xi\in\R_\mp,
 \end{split}\end{align} 
 for each $z\in\C$. 
 It follows from a version of Gronwall's inequality (see, for example, \cite[Lemma~1.3]{be89}, \cite[Lemma~A.1]{MeasureSL}) 
% In the plus case, take the right-hand limit of~\eqref{eqnCanSysString}, note that the integrands are continuous, apply the comment after \cite[Lemma~A.1]{MeasureSL} and take left-hand limits in the inequality again.  
 that the function $Y_\pm(z,\redot)$ satisfies the estimate
 \begin{align}\label{eqnYpmbound}
  \| Y_\pm(z,\xi) \| \leq \E^{\Lambda_\pm(\xi) |z|}, \quad \xi\in\R_\mp, 
 \end{align}
 where $\|\cdot\|$ denotes the max norm and the non-negative function $\Lambda_\pm$ is given by 
 \begin{align*}
  \mp \Lambda_\pm(\xi) = 2  \int_0^\xi \max\left(1,\,\Ar_\pm(s)^2\right) ds + \int_0^\xi d\beta_\pm, \quad \xi\in\R_\mp.
 \end{align*}
 Plugging this estimate back into the integral equation~\eqref{eqnCanSysString}, we furthermore get
 \begin{align}\label{eqnYpmAsymzero}
  \left\| Y_\pm(z,\xi) - \begin{pmatrix} 1 & 0 \\ 0 & 1 \end{pmatrix} \right\| \leq |z| \Lambda_\pm(\xi) \E^{\Lambda_\pm(\xi) |z|}, \quad \xi\in\R_\mp.
 \end{align}
 For the top-right entry of $Y_\pm(z,\redot)$, we are actually able to sharpen the estimate to  
 \begin{align*}
  \left| Y_{\pm,12}(z,\xi)\right|  \leq |z|\E^{\Lambda_\pm(\xi) |z|} \biggl( \mp \int_0^\xi |\Ar_\pm(s)|ds \mp\xi \biggr), \quad \xi\in\R_\mp.
 \end{align*}
 Upon plugging this back into the integral equation~\eqref{eqnCanSysString} one more time, we obtain 
 \begin{align}\label{eqnAsymYpm12}
  \left| Y_{\pm,12}(z,\xi) + z\,\xi\right| \leq 2 |z|^2 |\xi| \Lambda_\pm(\xi) \E^{\Lambda_\pm(\xi)|z|}, \quad \xi\in\R_\mp. 
 \end{align}

 Next, we introduce the matrix function $U_\pm(z,\redot)$ on $\R$ by 
 \begin{align*}
  U_\pm(z,x) = \begin{pmatrix} \E^{\pm\frac{x}{2}} & 0 \\ 0 & \E^{\mp\frac{x}{2}} \end{pmatrix} Y_\pm(z,\eta_\pm^{-1}(x)), \quad x\in\R.
 \end{align*}
 From the integral equation~\eqref{eqnCanSysString}, a substitution (use, for example, \cite[Theorem~3.6.1 and Corollary~3.7.2]{bo07}) as well as the definitions of $\Ar_\pm$ and $\beta_\pm$, one sees that 
 \begin{align*}
  \left. Y_\pm(z,\eta_\pm^{-1}(\ledot))\right|_{x}^{y} & = z\int_x^y \begin{pmatrix} -\alpha_\pm(s) & - \E^{\mp s} \\ \alpha_\pm(s)^2 \E^{\pm s} & \alpha_\pm(s) \end{pmatrix} Y_\pm(z,\eta_\pm^{-1}(s))ds \\
     & \qquad\qquad\qquad + z \int_x^y \begin{pmatrix} 0 & 0 \\ \E^{\pm s} & 0 \end{pmatrix} Y_\pm(z,\eta_\pm^{-1}(s))d\dip(s)
 \end{align*}
 for all $x$, $y\in\R$. 
 Upon employing the integration by parts formula~\eqref{eqnPI}, one sees that $U_\pm(z,\redot)$ is a solution of the integral equation~\eqref{eqnEquCanSysInt} and thus the system~\eqref{eqnEquCanSys}.

 Now under the additional assumption that $z$ is non-zero, let $\theta_\pm(z,\redot)$ and $\phi_\pm(z,\redot)$ be the solutions of the differential equation~\eqref{eqnDE} such that   
 \begin{align*}
  U_\pm(z,x) = \pm \begin{pmatrix} - \theta_\pm(z,x) & z\phi_\pm(z,x) \\ 
       \frac{1}{z} \theta_\pm^\qpm(z,x) & -\phi_\pm^\qpm(z,x) \end{pmatrix}, \quad x\in\R,
 \end{align*}
 guaranteed to exist by Lemma~\ref{lemEquCanSys}. 
 For future purposes, let us note the bounds 
 \begin{align}\begin{split}\label{eqnBoundThetaPhi}
   & \bigl| \theta_\pm(z,x) \E^{\mp\frac{x}{2}} \bigr|,\, \bigl| \theta_\pm^\qd(z,x) \E^{\mp\frac{x}{2}} \bigr|,\, \bigl| \phi_\pm(z,x)\E^{\pm\frac{x}{2}} \bigr|,\, \bigl| \phi_\pm^\qd(z,x) \E^{\pm\frac{x}{2}} \bigr| \\
   & \qquad\qquad\qquad\qquad\qquad\qquad\qquad\qquad\qquad\qquad \leq2  \E^{4\Lambda_\pm\left(\mp\E^{\mp x}\right) |z|}, \quad x\in\R, 
 \end{split}\end{align}
 that follow from the estimates~\eqref{eqnYpmbound} and~\eqref{eqnAsymYpm12} for $Y_\pm(z,\redot)$, also using that 
 \begin{align*}
% To get the first equality, apply the integration by parts formula to $u(x)\E^{\pm x} = ... $
  |u(x)| & = \left| \int_x^{\pm\infty} \E^{-|x-s|} \alpha_\pm(s)ds \right| % \leq \mp \int_0^{\mp\E^{\mp x}} |\Ar_\pm(s)|ds 
                    \leq \frac{1}{2}  \Lambda_\pm\left(\mp\E^{\mp x}\right), \quad x\in\R. 
 \end{align*}
 Furthermore, the inequalities in~\eqref{eqnYpmAsymzero} and~\eqref{eqnAsymYpm12} turn into  
 \begin{align}\begin{split}\label{eqnphipmasymest}
  & \bigl| \theta_\pm(z,x) \E^{\mp\frac{x}{2}} \pm 1\bigr|,\, \bigl| \tfrac{1}{z} \theta_\pm^\qpm(z,x) \E^{\pm\frac{x}{2}} \bigr|,\, \bigl| \phi_\pm(z,x) \E^{\pm\frac{x}{2}} - 1\bigr|,\,  \bigl| \phi_\pm^\qpm(z,x) \E^{\pm\frac{x}{2}} \pm 1\bigr| \\
    & \qquad\qquad\qquad\qquad\qquad\qquad\qquad \leq 2 |z| \Lambda_\pm\left(\mp\E^{\mp x}\right)\E^{\Lambda_\pm\left(\mp\E^{\mp x}\right) |z|}, \quad x\in\R.  \end{split} 
 \end{align}
 In particular, we see that $\phi_\pm(z,\redot)$ has the required asymptotics~\eqref{eqnphiasym} as well as
 \begin{align*}
    \theta_\pm(z,x) & \sim \mp \E^{\pm\frac{x}{2}}, \qquad x \rightarrow \pm\infty, 
 \end{align*}
 which implies that $\phi_\pm(z,\redot)$ is indeed uniquely determined by the asymptotics in~\eqref{eqnphiasym}. 
 Finally, the fact that the derivative of $\phi_\pm(z,\redot)$ is integrable and square integrable near $\pm\infty$ follows from the last bound in~\eqref{eqnBoundThetaPhi}. 
 It remains to set
 \begin{align*}
  \phi_\pm(0,x) & = \E^{\mp\frac{x}{2}}, & \theta_\pm(0,x) & = \mp\E^{\pm\frac{x}{2}},  
 \end{align*}
 for all $x\in\R$ and note that the claim is obvious in the case when $z$ is zero. 
\end{proof}

 For every $z\in\C$ we introduce the complex number $W(z)$ as the Wronskian of the two solutions $\phi_+(z,\redot)$ and $\phi_-(z,\redot)$, that is, in such a way that 
 \begin{align}
  W(z) = \phi_+(z,x)\phi_-'(z,x) - \phi_+'(z,x)\phi_-(z,x)
 \end{align}
 for almost all $x\in\R$; see Corollary~\ref{corWronskian}.  
 It follows readily that the set of zeros of $W$ coincides with the spectrum $\sigma$.  
 Thus, the next result implies that $\sigma$ is a discrete set of nonzero reals with convergence exponent at most one. 
 
 \begin{corollary}\label{corEvCar}
   The functions $\phi_\pm(\ledot,x)$ and $W$ are real entire of Cartwright class with only nonzero and real roots for each fixed $x\in\R$.
  Moreover, the function 
  \begin{align}\label{eqnGreens}
   \frac{z\phi_-(z,x) \phi_+(z,x)}{W(z)}, \quad z\in\C\backslash\R,
  \end{align}
  is a meromorphic Herglotz--Nevanlinna function. 
 \end{corollary}
 
 \begin{proof}
  Since the entries of $Y_\pm(\ledot,\xi)$ are real entire of Cartwright class with only real roots (cf.\ \cite[Section~1]{krla14}), the same holds for $\phi_\pm(\ledot,x)$ upon noting that 
  \begin{align*}
    \lim_{z\rightarrow 0} \phi_\pm(z,x) = \pm \E^{\pm\frac{x}{2}} \lim_{z\rightarrow0} \frac{Y_{\pm,12}\left(z,\mp\E^{\mp x}\right)}{z} = \E^{\mp\frac{x}{2}} = \phi_\pm(0,x).
  \end{align*}
  Because the function $\phi_\pm^\qd(\ledot,x)$ is real entire of Cartwright class as well, so is $W$ as  
  \begin{align*}
   W(z) & = \phi_+(z,x) \phi_-^\qd(z,x) - \phi_+^\qd(z,x) \phi_-(z,x), \quad z\in\C. 
  \end{align*}
  Next, we observe that the meromorphic function 
  \begin{align*}
   \pm\frac{\phi_\pm^\qd(z,x)}{z\phi_\pm(z,x)}, \quad z\in\C\backslash\R, 
  \end{align*}  
  is a Herglotz--Nevanlinna function. 
  In fact, to this end one just needs to evaluate its imaginary part and use~\eqref{eqnSIP} 
% to obtain
%   \begin{align*}
%      & \left.\frac{z\phi_\pm(z,\redot) \phi_\pm^\qd(z^\ast,\redot) - \phi_\pm^\qd(z,\redot) z^\ast \phi_\pm(z^\ast,\redot)}{z-z^\ast}\right|_x^y \\ 
%      & \qquad\quad = \int_x^{y} |\phi_\pm'(z,s)|^2 ds + \frac{1}{4} \int_x^{y} |\phi_\pm(z,s)|^2 ds + \int_x^{y} |z\phi_\pm(z,s)|^2 d\dip(s), \quad y\in\R,
%   \end{align*}
%   for all $z\in\C\backslash\R$, 
 as well as the vanishing asymptotics of the solution $\phi_\pm(z,\redot)$ near $\pm\infty$. 
% (that is, letting $y\rightarrow\pm\infty$ and noting that the quotient vanishes at $\pm\infty$)
  Since we may write 
  \begin{align*}
    - \frac{W(z)}{z\phi_-(z,x) \phi_+(z,x)} = \frac{\phi_+^\qd(z,x)}{z\phi_+(z,x)} - \frac{\phi_-^\qd(z,x)}{z\phi_-(z,x)}, \quad z\in\C\backslash\R,
  \end{align*}  
  this shows that the function in~\eqref{eqnGreens} is a Herglotz--Nevanlinna function as well. 
  In particular, this guarantees that $W$ has only nonzero and real roots indeed.  
  \end{proof}
 
 \begin{remark}\label{remETW}
 Although we will not prove this here, let us mention that it is possible to show that the exponential type of the entire function $W$ is simply given by 
  \begin{align}\label{eqnETW}
    \int_\R \rho(x) dx,
  \end{align}
  where $\rho$ is the square root of the Radon--Nikod\'ym derivative of the absolutely continuous part of the Borel measure $\dip$ (with respect to the Lebesgue measure). 
% Because the Wronskian $W$ will be invariant under the conservative Camassa--Holm flow, the functional in~\eqref{eqnETW} will reappear as a conserved quantity. 
 \end{remark}

 % % % % % % % % % % % % % % % % % % % % % % % % % % 
% \section{Coupling and norming constants}\label{secNC}
 % % % % % % % % % % % % % % % % % % % % % % % % % %
 
 As the spectrum alone will not be enough, we need to introduce further spectral quantities.   
 Since the solutions $\phi_+(\lambda,\redot)$ and $\phi_-(\lambda,\redot)$ are linearly dependent for each fixed $\lambda\in\sigma$ in view of Corollary~\ref{corWronskian}, there is a unique nonzero $c_{\lambda}\in\R$ such that 
 \begin{align}\label{eqnCoup}
  \phi_-(\lambda,x) = c_{\lambda} \phi_+(\lambda,x), \quad x\in\R, 
 \end{align}
 which will be referred to as the coupling constant associated with $\lambda$. 
 It will often be more convenient to work with the logarithmic coupling constant $\kappa_\lambda$, defined by 
 \begin{align}
  \kappa_\lambda = \ln|c_\lambda|,  
 \end{align}
 instead. 
 We will see below that one can always recover (the sign of) the coupling constant $c_\lambda$ from the quantity $\kappa_\lambda$, provided the spectrum $\sigma$ is known as well. 
 Furthermore, we introduce the right/left (modified) norming constant $\gamma_{\lambda,\pm}^2\in\R$ via 
 \begin{align}\label{eqnNormConstDef}
  \gamma_{\lambda,\pm}^2 = \omega\big(\phi_\pm(\lambda,\redot)^2\big) + 2\lambda \int_\R \phi_\pm(\lambda,x)^2 d\dip(x).
 \end{align}
  Upon employing~\eqref{eqnSIP} and the asymptotics from Theorem~\ref{thmThetaPhi}, one sees that 
 \begin{align}\label{eqnNormConstH1}
  \lambda\gamma_{\lambda,\pm}^2 = \int_\R \phi_\pm'(\lambda,x)^2 dx +  \frac{1}{4} \int_\R \phi_\pm(\lambda,x)^2 dx + \lambda^2 \int_\R \phi_\pm(\lambda,x)^2 d\dip(x) > 0.
 \end{align} 
 The following result gives a relation between all our spectral quantities. 

 \begin{lemma}\label{lemWderlam}
  For each $\lambda\in\sigma$ we have the relation 
 \begin{align}\label{eqnWlam}
 - \dot{W}(\lambda) = c_{\lambda}^{\pm1} \gamma_{\lambda,\pm}^{2}.
\end{align}
\end{lemma}

\begin{proof}
Let us fix an arbitrary $z\in\C$ and introduce the function
\begin{align*}
 W_\pm(z,x) = \dot{\phi}_\pm(z,x) \phi_\mp^\qd(z,x) - \dot{\phi}^\qd_\pm(z,x)\phi_\mp(z,x), \quad x\in\R, 
\end{align*} 
where the dot denotes differentiation with respect to the spectral parameter. 
From the bounds in~\eqref{eqnphipmasymest} and Cauchy's integral formula, we obtain the estimates 
\begin{align*}
  \bigl| \dot{\phi}_\pm(z,x)\E^{\pm\frac{x}{2}}\bigr|,\, \bigl| \dot{\phi}_\pm^\qpm(z,x) \E^{\pm\frac{x}{2}} \bigr| & \leq 2 (|z|+1) \Lambda_\pm\left(\mp\E^{\mp x}\right) \E^{\Lambda_\pm\left(\mp\E^{\mp x}\right) (|z|+1)}, \quad x\in\R.
\end{align*}
 Upon writing $\phi_\mp(z,\redot)$ as a linear combination of $\phi_\pm(z,\redot)$ and $\theta_\pm(z,\redot)$, we see from the bounds in~\eqref{eqnBoundThetaPhi} that  the function $W_\pm(z,x)$ tends to zero as $x\rightarrow\pm\infty$. 
 Next, we note that from~\eqref{eqnfqpm} and~\eqref{eqnDEweakderiv}  one gets 
\begin{align*}
%  \phi_\pm(z,\redot) \bigr|_x^y & = \int_x^y \phi_\pm^\qd(z,s)ds + z \int_x^y u'(s) \phi_\pm(z,s)ds, \\ 
  \phi_\pm(z,\redot) \bigr|_x^y & = \int_x^y z u'(s) \phi_\pm(z,s) + \phi_\pm^\qd(z,s) \, ds, \\ 
%  \phi_\pm^\qd(z,\redot) \bigr|_x^y & = \frac{1}{4} \int_x^y \phi_\pm(z,s)ds - z\int_x^y u(s)\phi_\pm(z,s) + u'(s)\phi_\pm^\qd(z,s)ds \\
%                                                              & \qquad - z^2 \int_x^y u'(s)^2 \phi_\pm(z,s)ds - z^2 \int_x^y \phi_\pm(z,s) d\dip(s), \\
 \phi_\pm^\qd(z,\redot) \bigr|_x^y & = \int_x^y \left( \frac{1}{4}  - z u(s) - z^2 u'(s)^2\right) \phi_\pm(z,s) -z u'(s)\phi_\pm^\qd(z,s) \, ds \\
                                                              & \qquad  - z^2 \int_x^y \phi_\pm(z,s) d\dip(s), 
\end{align*}
for all $x$, $y\in\R$, and after differentiating with respect to $z$ also 
\begin{align*}
%  \dot{\phi}_\pm(z,\redot) \bigr|_x^y & = \int_x^y \dot{\phi}_\pm^\qd(z,s)ds + \int_x^y u'(s) \phi_\pm(z,s)ds + z \int_x^y u'(s) \dot{\phi}_\pm(z,s)ds, \\ 
  \dot{\phi}_\pm(z,\redot) \bigr|_x^y & = \int_x^y z u'(s) \dot{\phi}_\pm(z,s) + \dot{\phi}_\pm^\qd(z,s) +  u'(s) \phi_\pm(z,s) \, ds, \\ 
%  \dot{\phi}_\pm^\qd(z,\redot) \bigr|_x^y & = \frac{1}{4} \int_x^y \dot{\phi}_\pm(z,s)ds - \int_x^y u(s)\phi_\pm(z,s) + u'(s)\phi_\pm^\qd(z,s)ds \\
%                                                               & \qquad - z\int_x^y u(s) \dot{\phi}_\pm(z,s) + u'(s) \dot{\phi}_\pm^\qd(z,s)ds \\
%                                                              & \qquad - 2 z \int_x^y u'(s)^2 \phi_\pm(z,s)ds - z^2 \int_x^y u'(s)^2 \dot{\phi}_\pm(z,s)ds \\
%                                                              & \qquad - 2 z \int_x^y \phi_\pm(z,s) d\dip(s) - z^2 \int_x^y \dot{\phi}_\pm(z,s) d\dip(s). \\
  \dot{\phi}_\pm^\qd(z,\redot) \bigr|_x^y & =  \int_x^y \left( \frac{1}{4}  - z u(s) - z^2 u'(s)^2\right) \dot{\phi}_\pm(z,s) - z  u'(s) \dot{\phi}_\pm^\qd(z,s) \, ds  \\
                                                               & \qquad - \int_x^y \left(u(s) + 2 zu'(s)^2\right)\phi_\pm(z,s) + u'(s)\phi_\pm^\qd(z,s) \, ds \\
                                                              & \qquad - 2 z \int_x^y \phi_\pm(z,s) d\dip(s) - z^2 \int_x^y \dot{\phi}_\pm(z,s) d\dip(s).
\end{align*}
In conjunction with the integration by parts formula~\eqref{eqnPI}, this gives 
\begin{align*}
 \left. W_\pm(z,\redot) \right|_{x}^{y} & = \int_x^y u(s) \phi_+(z,s) \phi_-(z,s) +  u'(s)\left(\phi_+(z,\redot)\phi_-(z,\redot)\right)'(s)\,ds \\
    & \qquad + 2 z \int_x^y \phi_+(z,s) \phi_-(z,s) d\dip(s).
\end{align*}
Upon letting $y\rightarrow\pm\infty$ in this equation, we obtain   
\begin{align}\begin{split}\label{eqnWdot}
 - \dot{W}(z) & = W_-(z,x) - W_+(z,x) \\ & = \int_\R u(s) \phi_+(z,s) \phi_-(z,s) + u'(s)\left(\phi_+(z,\redot)\phi_-(z,\redot)\right)'(s) \, ds \\
                & \qquad + 2z \int_\R \phi_+(z,s)\phi_-(z,s)d\dip(s),
\end{split}\end{align}
which readily yields the claimed identity (one should also note that the functions $u$ and $u'$ are integrable due to the growth restriction in~\eqref{eqnDedef} 
% Indeed we have 
%\begin{align}
% |u(x)|^2 \leq \left| \int_x^{\pm\infty} \E^{-|x-s|} \alpha_\pm(s)ds  \right|^2 \leq \frac{\E^{\mp x}}{3} \int_\R \E^{\pm s} \alpha_\pm(s)^2 ds, \quad x\in\R.  
%\end{align}
and thus the bounds in~\eqref{eqnBoundThetaPhi} guarantee that all integrals exist indeed).
\end{proof}

 In particular, the previous result shows that all zeros of the entire function $W$ are simple. 
 As a Cartwright class function, the Wronskian $W$ thus admits a product representation (see, for example, \cite[Section~17.2]{le96}) of the form
 \begin{align}\label{eqnWprodrep}
  W(z) = \cprod_{\lambda\in\sigma} \biggl(1-\frac{z}{\lambda}\biggr), \quad z\in\C. 
 \end{align}
 Upon invoking the identities in~\eqref{eqncprodatzero}, this fact allows us to read off trace formulas for the spectrum $\sigma$ from the derivatives of the function $W$ at zero. 
% This fact allows us to read off trace formulas from the derivatives of $W$ at zero using~\eqref{eqncprodatzero}. 
% Since the spectrum $\sigma$ will be invariant under the conservative Camassa--Holm flow, these trace formulas will give rise to conserved quantities.

\begin{proposition}\label{propTrF}
 The first two trace formulas are: 
 \begin{align}
  \csum_{\lambda\in\sigma} \frac{1}{\lambda} & = \int_\R u(x)dx, & \frac{1}{2} \sum_{\lambda\in\sigma} \frac{1}{\lambda^2} & = \int_\R d\mu. % \int_\R |u(x)|^2 + |u'(x)|^2 dx + \int_\R d\dip.
 \end{align}
\end{proposition}

\begin{proof}
 In view of~\eqref{eqncprodatzero}, we immediately obtain the first identity from~\eqref{eqnWdot} and we are left to compute the second derivative of $W$ at zero. 
 For this purpose, we first introduce the entire function $V_\pm$ by 
  \begin{align*}
     V_\pm(z) & =  \theta_\pm(z,x) \phi_\mp^\qd(z,x) - \theta_\pm^\qd(z,x) \phi_\mp(z,x), \quad x\in\R,~z\in\C,
 \end{align*}  
  so that we may write  
 \begin{align}\label{eqnphipmasmp}
  \phi_\pm(z,x) = \pm W(z) \theta_{\mp}(z,x) + V_\mp(z) \phi_\mp(z,x), \quad x\in\R,~ z\in\C. 
 \end{align}
 In conjunction with~\eqref{eqnBoundThetaPhi} and~\eqref{eqnphipmasymest}, this allows us to estimate  
 \begin{align}\begin{split}\label{eqnPhipmatmp}
   & \bigl| \phi_\pm(z,x)\E^{\pm\frac{x}{2}} - W(z) \bigr|,\,  \bigl| \phi_\pm^\qpm(z,x) \E^{\pm\frac{x}{2}} \pm W(z)\bigr| \\
   & \qquad  \leq  4 \E^{5\Lambda_\mp\left(\pm\E^{\pm x}\right) |z|}\left(|zW(z)| \Lambda_\mp\left(\pm\E^{\pm x}\right) \left(|z|\E^{\pm x}+1\right) + |V_\mp(z)| \E^{\pm x}\right)
  \end{split} \end{align}
  for all $x\in\R$ and $z\in\C$. 
 From this we see that each of the three products 
 \begin{align*}
   & \phi_+(z,x)\phi_-(z,x), & &\phi_+(z,x) \phi_-^\qd(z,x), & & \phi_+^\qd(z,x) \phi_-(z,x), 
 \end{align*}
 is bounded uniformly in $x\in\R$ and locally uniformly in $z\in\C$. 
 Thus, we may differentiate~\eqref{eqnWdot} under the integral and in order to evaluate the derivative at zero, we first note that 
 \begin{align}\label{eqnYpmderatzero}
  \dot{Y}_{\pm}(0,\xi) & = \int_0^{\xi} \begin{pmatrix} -\Ar_\pm(s) & -1 \\ \Ar_\pm(s)^2 & \Ar_\pm(s) \end{pmatrix} ds + \int_0^\xi \begin{pmatrix} 0 & 0 \\ 1 & 0 \end{pmatrix} d\beta_\pm(s), \quad \xi\in\R_\mp,  
  \end{align}
 which follows from~\eqref{eqnCanSysString}. 
 Moreover, for the top-right entry we even have 
 \begin{align*}
   \ddot{Y}_{\pm,12}(0,\xi) & =  2\xi \int_0^\xi \Ar_\pm(s)ds - 4 \int_0^\xi \int_0^s \Ar_\pm(r)dr\, ds, \quad \xi\in\R_\mp. 
 \end{align*}  
 After performing substitutions, this turns into the identities 
 \begin{align*}
%  \dot{\phi}_\pm(0,x)& =  \E^{\mp\frac{x}{2}} \left(u(x) \mp \int_x^{\pm\infty} u(s)ds\right), \\
%  \dot{\phi}_\pm^\qd(0,x)  & = \pm\frac{1}{2} \E^{\mp\frac{x}{2}} \left(u(x) \pm \int_x^{\pm\infty} u(s)ds\right), \\
    \dot{\phi}_\pm(0,x)& =  \E^{\mp\frac{x}{2}} \int_x^{\pm\infty} \alpha_\mp(s)ds, &
  \dot{\phi}_\pm^\qd(0,x) &  = \pm\frac{1}{2} \E^{\mp\frac{x}{2}}  \int_x^{\pm\infty} \alpha_\pm(s)ds,  
 \end{align*}
 for all $x\in\R$. 
 Plugging them into the differentiated integral in~\eqref{eqnWdot}, we arrive at
 \begin{align*}
  - \ddot{W}(0) = 2 \int_\R u(x)^2 + u'(x)^2\, dx + 2 \int_\R d\dip - \left(\int_\R u(x)dx\right)^2, 
 \end{align*}
 which proves the remaining claim in view of~\eqref{eqncprodatzero}.  
\end{proof}

Before we are able to solve the inverse spectral problem, we need to derive one more necessary condition. 
More precisely, we will show that the growth restriction on the pair $(u,\mu)$ in~\eqref{eqnDedef} implies certain asymptotic behavior of the spectral data. %coupling and norming constants. 

 \begin{proposition}\label{propMatZ}   
 We have the identity 
 \begin{align}\label{eqnLCPars}
   \int_\R \E^{\pm x} \left(u'(x)\mp u(x)\right)^2 dx + \int_\R \E^{\pm x} d\dip(x) = \sum_{\lambda\in\sigma} \frac{1}{\lambda^{2}} \frac{1}{\lambda\gamma_{\lambda,\pm}^{2}}.
  \end{align}
 \end{proposition}

\begin{proof}
 To begin with, we introduce the function $m_\pm$ by
 \begin{align*}
  m_\pm(z) = \frac{V_\pm(z)}{zW(z)}, \quad z\in\C\backslash\R. 
 \end{align*}
 From the estimates in~\eqref{eqnPhipmatmp}, we see that  
 \begin{align*}
   \phi_\pm(z,x) \E^{\pm\frac{x}{2}} & \rightarrow W(z), & \phi_\pm^\qpm(z,x) \E^{\pm\frac{x}{2}} & \rightarrow \mp W(z), & x & \rightarrow\mp\infty,
 \end{align*}
 where the convergence is locally uniform in $z\in\C$. 
 In much the same manner as above (upon writing $\theta_\pm$ as in~\eqref{eqnphipmasmp} and using~\eqref{eqnBoundThetaPhi} as well as~\eqref{eqnphipmasymest} to estimate), one obtains inequalities similar to~\eqref{eqnPhipmatmp} for $\theta_\pm$ and concludes that 
 \begin{align*}
  \theta_\pm(z,x) \E^{\pm\frac{x}{2}} & \rightarrow \pm V_\pm(z), &  \theta_\pm^\qpm(z,x) \E^{\pm\frac{x}{2}} & \rightarrow - V_\pm(z), & x & \rightarrow\mp\infty,
 \end{align*} 
 where the convergence is again locally uniform in $z\in\C$. 
 Thus, we may write  
  \begin{align}\label{eqnmpmaslim}
  m_\pm(z) & = \pm \lim_{x\rightarrow\mp\infty} \frac{\theta_\pm(z,x)}{z\phi_\pm(z,x)} = \mp \lim_{\xi\rightarrow\mp\infty} \frac{Y_{\pm,11}(z,\xi)}{Y_{\pm,12}(z,\xi)},\quad z\in\C\backslash\R,
 \end{align}
 that is, the function $m_\pm$ is the Weyl--Titchmarsh function for~\eqref{eqnCanSysString}; cf.\ \cite[Equation~(6.2)]{IndefiniteString}.
 In particular, it is a meromorphic Herglotz--Nevanlinna function \cite[Lemma~5.1]{IndefiniteString} with simple poles at all points $\lambda\in\sigma$ (note that zero is not a pole since $V_\pm$ vanishes at zero) with residues given by  
 \begin{align*}
  \res{\lambda} m_\pm = \frac{V_\pm(\lambda)}{\lambda\dot{W}(\lambda)} = - \frac{1}{\lambda\gamma_{\lambda,\pm}^{2}}, \quad \lambda\in\sigma,
 \end{align*}
 in view of Lemma~\ref{lemWderlam}. 
 After a substitution, we obtain from~\eqref{eqnYpmderatzero} that 
 \begin{align*}
%  \dot{\theta}_\pm(0,x) & = \mp \E^{\pm\frac{x}{2}} \left( u(x) \pm \int_x^{\pm\infty} u(s)ds\right), \\
%  \dot{\theta}_\pm^\qd(0,x) & = \frac{1}{2} \E^{\pm\frac{x}{2}} \left(u(x) \mp \int_x^{\pm\infty} u(s)ds \right). \\
    \dot{\theta}_\pm(0,x) & = \mp \E^{\pm\frac{x}{2}}  \int_x^{\pm\infty} \alpha_\pm(s)ds, &  \dot{\theta}_\pm^\qd(0,x) & = \frac{1}{2} \E^{\pm\frac{x}{2}} \int_x^{\pm\infty} \alpha_\mp(s)ds, 
 \end{align*}
 for every $x\in\R$, which implies that $\dot{V}_\pm(0) = 0$. 
 Upon recalling the integral representation formula for Herglotz--Nevanlinna functions (see, for example, \cite[Section~5.3]{roro94}), this guarantees that the function $m_\pm$ admits the representation (use also \cite[Lemma~7.1]{IndefiniteString} to conclude that there is no linear term present)
 \begin{align}\label{eqnmpmIntRep}
  m_\pm(z) & =  \sum_{\lambda\in\sigma} \frac{z}{\lambda(\lambda-z)}  \frac{1}{\lambda\gamma_{\lambda,\pm}^{2}}, \quad z\in\C\backslash\R. 
 \end{align}
 Clearly, we may just as well write the function $m_\pm$ as the limit 
 \begin{align}\label{eqnmpmaslim2}
   m_\pm(z) & = \pm \lim_{x\rightarrow\mp\infty} \frac{\theta_\pm^\qpm(z,x)}{z\phi_\pm^\qpm(z,x)} = \mp \lim_{\xi\rightarrow\mp\infty} \frac{Y_{\pm,21}(z,\xi)}{Y_{\pm,22}(z,\xi)},\quad z\in\C\backslash\R.
 \end{align}
 Now we observe that~\eqref{eqnYpmderatzero} gives for every $\xi\in\R_\mp$ the expansion 
 \begin{align*}
  \frac{Y_{\pm,21}(z,\xi)}{Y_{\pm,22}(z,\xi)} = z \int_0^\xi \Ar_\pm(s)^2 ds + z \int_0^\xi d\beta_\pm + \OO(z^2), \qquad z\rightarrow 0.
 \end{align*}
 Since the convergence in~\eqref{eqnmpmaslim2} is uniform for all $z\in\C\backslash\R$ that lie in a small neighborhood of zero, we infer, after a substitution, that $m_\pm$ has the expansion 
 \begin{align*}
  m_\pm(z) = z \int_\R \E^{\pm x} \left(u'(x)\mp u(x)\right)^2 dx + z \int_\R \E^{\pm x} d\dip(x) + \OO(z^2), \qquad z\rightarrow0. 
 \end{align*}
 Upon differentiating~\eqref{eqnmpmIntRep} and letting $z\rightarrow0$, we obtain the identity in~\eqref{eqnLCPars}.  
 \end{proof}

 In this context, let us also mention the following result which characterizes the subclass of $\De$ that gives rise to purely positive/negative spectrum.  

\begin{proposition}\label{propDefinite}
 The spectrum $\sigma$ is positive/negative if and only if the Borel measure $\dip$ vanishes identically and the distribution $\omega$ is non-negative/non-positive. 
 In this case, the distribution $\omega$ can be represented by a non-negative/non-positive finite Borel measure on $\R$ (for simplicity denoted with $\omega$ as well) and 
 \begin{align}\label{eqnTFPN}
  \sum_{\lambda\in\sigma} \frac{1}{\lambda} = \int_\R d\omega. 
 \end{align}
\end{proposition}

\begin{proof}
 If $\dip$ vanishes identically and $\omega$ is non-negative/non-positive, then~\eqref{eqnNormConstDef} and~\eqref{eqnNormConstH1} show that the spectrum is positive/negative. 
 Conversely, if the spectrum is positive/negative, then~\cite[Lemma~7.2]{IndefiniteString} shows that the measures $\beta_+$ and $\beta_-$ vanish identically as well as that the functions $\Ar_+$ and $\Ar_-$ have non-decreasing/non-increasing representatives. 
 Clearly, the measures $\beta_+$ and $\beta_-$ vanish identically if and only if so does $\dip$. 
 Given some $h\in H^1(\R)$ with compact support, we set 
 \begin{align}\label{eqnDefhpm}
    h_\pm(\xi) = \mp\xi h(\eta_\pm(\xi)), \quad \xi\in\R_\mp,
 \end{align}
 and note that 
 \begin{align*}
  \omega(h) =  -\frac{1}{2} \int_{\R_-} h_+'(\xi) \Ar_+(\xi) d\xi - \frac{1}{2} \int_{\R_+} h_-'(\xi) \Ar_-(\xi) d\xi.
 \end{align*}
 This shows that $\omega$ is a non-negative/non-positive distribution and therefore can be represented by a non-negative/non-positive Borel measure on $\R$. 
 Now for every $k\in\N$ let $h_k\in H^1(\R)$ be the piecewise linear function such that $h_k$ is equal to one on $[-k,k]$, equal to zero outside of $[-k-1,k+1]$ and linear in between. 
 Then from  
 \begin{align*}
  \int_\R h_k\, d\omega = \int_\R u(x) h_k(x)dx + \int_\R u'(x) h_k'(x)dx
 \end{align*}
 we see that $\omega$ is a finite measure as well as~\eqref{eqnTFPN} % u decays exponentially at $\pm\infty$. 
 upon letting $k\rightarrow\infty$.
\end{proof} 

An identity similar to~\eqref{eqnTFPN} also holds under the sole assumption that the distribution $\omega$ can be represented by a real-valued Borel measure on $\R$ (again denoted with $\omega$ as well) with finite total variation. 
More precisely, one has  the equality 
\begin{align}
 \csum_{\lambda\in\sigma} \frac{1}{\lambda} = \int_\R d\omega
\end{align}
in this case, which follows readily from the first trace formula in Proposition~\ref{propTrF}.

%%%%%%%%%%%%%%%%%%%%%%%%%%%%%%
\section{The inverse spectral problem}
%%%%%%%%%%%%%%%%%%%%%%%%%%%%%%

 We are now going to solve the corresponding inverse spectral problem for the class $\De$. 
 Let us point out that the given sufficient conditions on the spectral data are also necessary in view of Corollary~\ref{corEvCar} (in conjunction with well-known properties of entire functions of Cartwright class; for example, see \cite[Chapter~8]{bo54} or \cite[Theorem~17.2.1]{le96}) as well as Proposition~\ref{propMatZ} and Lemma~\ref{lemWderlam}.   
 Thus, we indeed obtain a complete characterization of all possible spectral data for the class $\De$.   
   
\begin{theorem}\label{thmIP}
 Let $\sigma$ be a discrete set of nonzero reals such that the limit\footnote{Recall that we denote with $n_\sigma(r)$ the number of all $\lambda\in\sigma$ with modulus not greater than $r$.}
 \begin{align}\label{eqnIPdens}
  \lim_{r\rightarrow\infty} \frac{n_\sigma(r)}{r} 
 \end{align} 
 exists in $[0,\infty)$ and such that the entire function $W$ is well-defined by\footnote{If the limit in~\eqref{eqnIPdens} exists in $[0,\infty)$, then this is the case if and only if the sum
  $\csum_{\lambda\in\sigma} \frac{1}{\lambda}$ exists.}
 \begin{align}\label{eqnIPW}
  W(z) = \cprod_{\lambda\in\sigma} \biggl(1-\frac{z}{\lambda}\biggr), \quad z\in\C.
 \end{align}
 Moreover, for each $\lambda\in\sigma$ let $\kappa_{\lambda}\in\R$ such that the sum
 \begin{align}\label{eqnIPlogccCond}
   \sum_{\lambda\in\sigma} \frac{1}{\lambda^2} \frac{\E^{|\kappa_{\lambda}|}}{|\lambda \dot{W}(\lambda)|} 
  \end{align}
 is finite. 
 Then there is a unique pair $(u,\mu)\in\De$ such that the associated spectrum coincides with $\sigma$ and the logarithmic coupling constants are $\kappa_\lambda$ for each $\lambda\in\sigma$. 
\end{theorem}

 \begin{proof}
    {\em Uniqueness.}
  Since the given spectral data uniquely determines the Weyl--Titchmarsh function $m_-$ in view of~\eqref{eqnmpmIntRep} and Lemma~\ref{lemWderlam}, the uniqueness part in \cite[Theorem~6.1]{IndefiniteString} shows that the function $\Ar_-$ and the Borel measure $\beta_-$ are uniquely determined as well. 
  It follows readily from the definition of $\beta_-$ in~\eqref{eqnDefbeta} that this also uniquely determines the Borel measure $\dip$. 
  Furthermore, because of
  \begin{align}\label{eqnuitofa}
   u(\eta_\pm(\xi)) = - \frac{1}{\xi} \int_0^\xi  \Ar_\pm(s) s\, ds, \quad \xi\in\R_\mp,  
  \end{align} 
  we conclude that the function $u$ is uniquely determined too and thus so is $\mu$. 
  
  {\em Existence.} 
  Let us introduce the meromorphic Herglotz--Nevanlinna function  
   \begin{align*}
    m(z) =  \sum_{\lambda\in\sigma} \frac{z}{\lambda(\lambda-z)}  \frac{\E^{-\kappa_\lambda}}{|\lambda\dot{W}(\lambda)|}, \quad z\in\C\backslash\R.
   \end{align*}
  From the existence part of \cite[Theorem~6.1]{IndefiniteString}, we obtain a real-valued and locally square integrable function $\Ar$ on $[0,\infty)$ and a non-negative Borel measure $\beta$ on $[0,\infty)$ with $\beta(\lbrace0\rbrace)=0$ such that the function $m$ is the Weyl--Titchmarsh function for  
   \begin{align}\begin{split}\label{eqnCanSysStringIP}
  Y(z,\xi) = \begin{pmatrix} 1 & 0 \\ 0 & 1 \end{pmatrix} & + z \int_0^\xi \begin{pmatrix} -\Ar(s) & -1 \\ \Ar(s)^2 & \Ar(s) \end{pmatrix} Y(z,s)ds \\ 
       & + z \int_{0}^{\xi} \begin{pmatrix} 0 & 0 \\ 1 & 0 \end{pmatrix} Y(z,s) d\beta(s), \quad \xi\in[0,\infty),
 \end{split}\end{align} 
 that is, if $Y(z,\redot)$ denotes the unique solution of the integral equation~\eqref{eqnCanSysStringIP} for each $z\in\C$, then the function $m$ is given by
 \begin{align*}
  m(z) = \lim_{\xi\rightarrow\infty} \frac{Y_{11}(z,\xi)}{Y_{12}(z,\xi)}, \quad z\in\C\backslash\R. 
 \end{align*}
  In order to state a required fact from \cite{IndefiniteString} in a concise way, we introduce the set $\Sigma\subseteq (0,\infty)$ that consists of all $\xi\in(0,\infty)$ such that the function $\Ar$ is not equal to a constant almost everywhere near $\xi$ or such that $\xi$ belongs to the topological support of $\beta$. % (that is, the set $\Sigma$ is the union of the support of the derivative of $\Ar$ and $\beta$). 
  Let $f_1$, $f_2\in L^2(0,\infty)$ be orthogonal to all $h\in L^2(0,\infty)$ that satisfy 
  \begin{align}\label{eqnhorth}
   \int_0^\xi h(s)ds = 0, \quad \xi\in\Sigma,
  \end{align}
   and functions $g_1$, $g_2$ on $[0,\infty)$ be square integrable with respect to $\beta$. 
  Upon setting
  \begin{align*}
   F_i(\lambda) = \int_0^\infty \frac{Y_{12}'(\lambda,\xi)}{\lambda} f_i(\xi)d\xi +\int_0^\infty Y_{12}(\lambda,\xi)g_i(\xi)d\beta(\xi), \quad \lambda\in\sigma, ~i=1,2,
  \end{align*}
  we have the identity  
  \begin{align}\label{eqnIPPars}
   \sum_{\lambda\in\sigma} F_1(\lambda) F_2(\lambda)^\ast \frac{\E^{-\kappa_\lambda}}{|\lambda\dot{W}(\lambda)|} = \int_0^\infty f_1(\xi)f_2(\xi)^\ast d\xi + \int_0^\infty g_1(\xi)g_2(\xi)^\ast d\beta(\xi)
  \end{align}
  since the spectral transform introduced in~\cite[Section~5]{IndefiniteString} is a partial isometry.

   Motivated by the relation~\eqref{eqnuitofa}, we define the real-valued function $u$ on $\R$ via 
  \begin{align}\label{eqnDefupm}
   u(\ln(\xi)) = - \frac{1}{\xi} \int_0^\xi  \Ar(s) s\, ds, \quad \xi\in(0,\infty),  
  \end{align} 
  so that $u$ belongs to $H^1_\loc(\R)$ % see \cite[Exercise~5.8.59, Corollary~5.4.4 and Corollary~3.7.2]{bo07}
  and satisfies 
  \begin{align*}
   \Ar(\xi) = \frac{-u'(\ln(\xi)) - u(\ln(\xi))}{\xi}
  \end{align*}
  for almost all $\xi\in(0,\infty)$. 
  Furthermore, we define the Borel measure $\mu$ on $\R$ by~\eqref{eqnDefdip}, where $\dip$ is the Borel measure on $\R$  such that 
  % Simply define them such that 
  %\begin{align*}
  %  \int_B d\dip = \int_B \E^{x} d\beta \circ\eta_-^{-1}(x)
  %\end{align*}
  %for every Borel set $B\subseteq\R$ such 
  that we have  
  \begin{align*}
   \beta(B) =  \int_{\ln(B)} \E^{-x} d\dip(x)
  \end{align*}
  for every Borel set $B\subseteq(0,\infty)$. 
  Since the function $\Ar$ is square integrable near zero and the Borel measure $\beta$ is finite near zero, it follows readily that 
  \begin{align*}
   \int_{-\infty}^{c} \E^{- x} \left(u'(x) + u(x)\right)^2 dx +   \int_{-\infty}^{c} \E^{-x} d\dip(x) < \infty
  \end{align*}
  for every $c\in\R$. 
  In particular, this guarantees that the function $u$ lies in $H^1(\R)$ near $-\infty$ (more precisely, one sees from~\eqref{eqnDefupm} % simply use Cauchy--Schwarz
  that $u$ decays exponentially near $-\infty$ and since $u' + u$ is clearly square integrable near $-\infty$, we conclude that $u'$ is as well) and that the measure $\mu$ is finite near $-\infty$. 
  We are now left to verify that the pair $(u,\mu)$ actually belongs to $\De$. 
  In fact, in this case it is readily seen that $\Ar$ coincides with the function $\Ar_-$ as introduced by~\eqref{eqnDefa} and that $\beta$ coincides with the Borel measure $\beta_-$ as introduced by~\eqref{eqnDefbeta}. 
  From~\eqref{eqnmpmaslim}, we then see that the function $m$ coincides with the Weyl--Titchmarsh function $m_-$ for this pair as introduced in the proof of Proposition~\ref{propMatZ}.
  Upon taking~\eqref{eqnmpmIntRep} and Lemma~\ref{lemWderlam} into account, one then may conclude that the pair $(u,\mu)$ indeed gives rise to the desired spectral data.  

  In order to show that the pair $(u,\mu)$ belongs to the set $\De$, let us first consider the special case when $\Ar$ is equal to some constant $\Ar_0\in\R$ almost everywhere near $\infty$ and $\beta$ vanishes near $\infty$. 
  From this we readily  infer that
  \begin{align*}
   Y(z,\xi) = \begin{pmatrix} 1- z\Ar_0(\xi-\xi_0) & -z(\xi-\xi_0) \\ z\Ar_0^2(\xi-\xi_0)  & 1 + z\Ar_0(\xi-\xi_0) \end{pmatrix} Y(z,\xi_0)
  \end{align*}
  as long as $\xi$ and $\xi_0$ are close enough to $\infty$. 
  Thus, the function $m$ is given  by 
  \begin{align*}
   m(z) = \frac{\Ar_0 Y_{11}(z,\xi_0) + Y_{21}(z,\xi_0)}{\Ar_0 Y_{21}(z,\xi_0) + Y_{22}(z,\xi_0)}, \quad z\in\C\backslash\R,
  \end{align*}
  and evaluating the limit as $z\rightarrow0$ in this equation shows that $\Ar_0$ has to be zero, which immediately implies that the pair $(u,\mu)$ belongs to $\De$. 
  In particular, this proves the claim in the case when the set $\sigma$ is finite (since then the set $\Sigma$ is finite as well).  % Each evaluation functional for a $x\in\Sigma$ is orthogonal to $\mul{T}$ and thus contained in the finite dimensional closure of $\dom{T}$.
  Thus, back in the general case,  for every $k\in\N$ we may find a $(u_k,\mu_k)\in\De$ such that the corresponding Weyl--Titchmarsh function $m_{k,-}$ is given by 
  \begin{align*}
   m_{k,-}(z) = \mathop{\sum_{\lambda\in\sigma}}_{|\lambda|\leq k} \frac{z}{\lambda(\lambda-z)}  \frac{\E^{-\kappa_\lambda}}{|\lambda\dot{W}(\lambda)|}, \quad z\in\C\backslash\R.
  \end{align*} 
  With the functions $\Ar_{k,-}$ defined as in~\eqref{eqnDefa}, we infer from \cite[Proposition~6.2]{IndefiniteString} that 
  \begin{align*}
   \lim_{k\rightarrow\infty} \int_{0}^\xi \Ar_{k,-}(s) ds = \int_0^\xi \Ar(s) ds,
  \end{align*}
  locally uniformly for all $\xi\in[0,\infty)$. 
  In view of~\eqref{eqnuitofa} and~\eqref{eqnDefupm}, this shows that the functions $u_k$ converge pointwise to $u$ and since they are uniformly bounded in $H^1(\R)$ by Proposition~\ref{propTrF}, a compactness argument shows that $u$ belongs to $H^1(\R)$. 
  
  To verify also the remaining growth restrictions on $(u,\mu)$ in the general case, we now may assume that there is an increasing sequence $x_n\rightarrow\infty$ such that $\xi_n=\E^{x_n}$ belongs to $\Sigma$ for every $n\in\N_0$. 
  We fix some $n\in\N_0$ and define the function $J_n$ by 
    \begin{align*}
   J_n(\zeta,z) = \frac{Z_1(z,\xi_n)Z_2(\zeta,\xi_n)^\ast - Z_2(z,\xi_n)Z_1(\zeta,\xi_n)^\ast}{z-\zeta^\ast}, \quad \zeta,\, z\in\C\backslash\R,
  \end{align*}  
 where $Z(z,\redot)$ is the Weyl solution given by 
  \begin{align*}
   Z(z,\xi) = Y(z,\xi) \begin{pmatrix}  zW(z) \\ -V(z) \end{pmatrix}, \quad \xi\in[0,\infty), 
  \end{align*}
  for every $z\in\C$, and $V$ is the entire function defined in such a way that
   \begin{align*}
   V(z) = zW(z) m(z), \quad z\in\C\backslash\R. 
   \end{align*}
  Because the function $Z(z,\redot)$ is a Weyl solution for every $z\in\C\backslash\R$, we note that $J_n$ can be rewritten as (see the proof of \cite[Lemma~5.1]{IndefiniteString})  
  \begin{align*}
   J_n(\zeta,z) =\int_{\xi_n}^\infty \frac{Z_1'(z,s)}{z} \frac{Z_1'(\zeta,s)^\ast}{\zeta^\ast} ds + \int_{\xi_n}^\infty Z_1(z,s) Z_1(\zeta,s)^\ast d\beta(s), \quad \zeta,\, z\in\C\backslash\R. 
  \end{align*}  
   In particular, this implies that the entire function $E_n$ defined by 
   \begin{align*}
    E_n(z) = z\psi(z,x_n) - \I \psi^\qp(z,x_n), \quad z\in\C,
   \end{align*}
   is a de Branges function, where $\psi(z,\redot)$ is the solution of the differential equation~\eqref{eqnDE}  so that  (compare the proof of Theorem~\ref{thmThetaPhi})
   \begin{align*}
   \begin{pmatrix} z\psi(z,x) \\ 
        -\psi^\qm(z,x) \end{pmatrix} = \begin{pmatrix} \E^{-\frac{x}{2}} & 0 \\ 0 & \E^{\frac{x}{2}} \end{pmatrix} Z(z,\E^{x}), \quad x\in\R,
 \end{align*}
   for each $z\in\C$, with the quantities $\omega$ and $\dip$ defined as in~\eqref{eqnDefomegaapp} and~\eqref{eqnDefdip}.  
   The reproducing kernel $K_n$ in the corresponding de Branges space $\mathcal{B}_n$ is given by 
   \begin{align*}
     K_n(\zeta,z) = J_n(\zeta,z) + \psi(z,x_n)\psi(\zeta,x_n)^\ast, \quad \zeta,\, z\in\C\backslash\R.
   \end{align*}
     % The de Branges space (and in particular the inner product is defined as in \cite{dBSing} so that $F(\zeta) = \spr{F}{K(\zeta,\redot)}$
   An integration by parts and the integral equation~\eqref{eqnCanSysStringIP} show that for every $z\in\C\backslash\R$ the function $Z_1'(z,\redot)$ restricted to $(\xi_n,\infty)$ is orthogonal to all functions $h\in L^2(0,\infty)$ that satisfy~\eqref{eqnhorth}. 
  Thus, the identity~\eqref{eqnIPPars} for the particular functions
  \begin{align*}
   f_i(s) & = \frac{1}{z_i} \begin{cases} Z_1(z_i,\xi_n)\E^{-x_n}, & s\in[0,\xi_n), \\  Z_1'(z_i,s), & s\in[\xi_n,\infty), \end{cases} & g_i(s) & = \begin{cases} 0, & s\in[0,\xi_n), \\ Z_1(z_1,s), & s\in[\xi_n,\infty), \end{cases}
  \end{align*}
  with $z_1=z^\ast$ and $z_2=\zeta^\ast$ gives (note that $Z_1(\lambda,\redot) = - V(\lambda)Y_{12}(\lambda,\redot)$ for all $\lambda\in\sigma$)
   \begin{align*}
    \sum_{\lambda\in\sigma} \frac{K_n(z,\lambda)}{V(\lambda)} \frac{K_n(\zeta,\lambda)^\ast}{V(\lambda)^\ast} \frac{\E^{-\kappa_\lambda}}{|\lambda\dot{W}(\lambda)|} & = K_n(z,\zeta) \\
          & = \spr{K_n(z,\redot)}{K_n(\zeta,\redot)}_{\mathcal{B}_n}, \quad \zeta,\, z\in\C\backslash\R.
   \end{align*}
  The values of $V$ on the set $\sigma$ are readily evaluated and we obtain % (note that $|V(\lambda)|=\E^{-\kappa_\lambda}$ for all $\lambda\in\sigma$) 
   \begin{align}\label{eqndBemb}
    \sum_{\lambda\in\sigma} |F(\lambda)|^2 \frac{\E^{\kappa_\lambda}}{|\lambda\dot{W}(\lambda)|} %= \frac{1}{\pi} \int_\R \frac{|F(\lambda)|^2}{|E_n(\lambda)|^2} d\lambda 
             = \|F\|_{\mathcal{B}_n}^2, \quad F\in\mathcal{B}_n,
   \end{align}
   after employing simple linearity, continuity and density arguments.
  
  Now let $H$ be a locally integrable, trace normed, real, symmetric and non-negative definite $2\times2$ matrix function on $[0,\infty)$ such that if $M(z,\redot)$ denotes the unique solution of the integral equation 
  \begin{align*}
   M(z,t) = \begin{pmatrix} 1 & 0 \\ 0 & 1 \end{pmatrix} + z \int_{0}^t \begin{pmatrix} 0 & -1 \\ 1 & 0 \end{pmatrix} H(s)M(z,s)ds, \quad t\in[0,\infty), 
  \end{align*}
  for every $z\in\C$, then we have    
  \begin{align*}
    \lim_{t\rightarrow\infty} \frac{M_{11}(z,t)}{M_{12}(z,t)} % =  \lim_{t\rightarrow\infty} \frac{M_{21}(z,t)}{M_{22}(z,t)}
         = \sum_{\lambda\in\sigma} \frac{z}{\lambda(\lambda-z)} \frac{\E^{\kappa_\lambda}}{|\lambda\dot{W}(\lambda)|}, \quad z\in\C\backslash\R.
  \end{align*}
  Such a function $H$ is guaranteed to exist by the solution of the inverse problem for canonical systems due to de Branges; see \cite[Theorem~XII]{dB61a}, \cite[Theorem~2.4]{wi14}. 
  Finiteness of the sum in~\eqref{eqnIPlogccCond} implies (see \cite[Lemma~5.5]{lawo02}, \cite[Theorem~6.17]{roro94} and note that existence of~\eqref{eqnIPdens} and~\eqref{eqnIPW} together establish \cite[Condition~(C2)]{lawo02}; cf.\ \cite[\S 8.3.6]{bo54}) 
%  Let $(\lambda_n)_{n\in\N}$ be an enumeration of $\sigma$ ordered according to non-decreasing modulus (we may assume that $\sigma$ is not finite). 
%  Then the limit
%  \begin{align*}
%   \lim_{N\rightarrow\infty} \sum_{n=1}^N \frac{1}{\lambda_n}, \quad \text{since}\quad \left| \sum_{n=1}^N \frac{1}{\lambda_n} - \sum_{\lambda\in\sigma, |\lambda|\leq |\lambda_N|} \frac{1}{\lambda} \right| \leq \frac{1}{\lambda_N}.
%  \end{align*}
%  By Kronecker's lemma, we conclude that 
%  \begin{align*}
%   \lim_{n\rightarrow\infty} \frac{n_+(|\lambda_n|) - n_-(|\lambda_n|)}{|\lambda_n|} = \lim_{n\rightarrow\infty}\frac{1}{|\lambda_n|} \sum_{k=1}^n \mathrm{sgn}(\lambda_k) = 0. 
%  \end{align*}
%  Together with~\eqref{eqnIPdens}, this shows existence of the limits
%  \begin{align*}
%   \lim_{n\rightarrow\infty} \frac{n_\pm(|\lambda_n|)}{|\lambda_n|}.
%  \end{align*}
  that the entire function $W$ belongs to the Cartwright class. 
  Moreover, since $m$ is a Herglotz--Nevanlinna function, so does  the function $V$ and thus also $E_n$ for every $n\in\N_0$. 
  In view of \cite[Theorem~VII]{dB62} in conjunction with~\eqref{eqndBemb} and upon employing \cite[Theorem~I]{dB60},  we obtain a $t_n\in(0,\infty)$ so that 
  \begin{align*}
    \begin{pmatrix} z\psi(z,x_n) \\ -\psi^\qp(z,x_n) \end{pmatrix} = \Gamma_n \begin{pmatrix} M_{12}(z,t_n) \\ M_{22}(z,t_n) \end{pmatrix}, \quad z\in\C,
  \end{align*}
  for a real $2\times2$ matrix $\Gamma_n$ with $\det\Gamma_n=-1$. 
  Evaluating at zero shows that 
  \begin{align*}
   \Gamma_n = \begin{pmatrix} -\E^{\frac{x_n}{2}} & 0 \\  \varepsilon_n & \E^{-\frac{x_n}{2}}  \end{pmatrix}
  \end{align*}
  for some $\varepsilon_n\in\R$. 
  We note that the sequence $t_n$ may be chosen in such a way that it is non-increasing because the sequence of de Branges spaces $\mathcal{B}_n$ is non-increasing with respect to inclusion.  
  Thus, for every $n\in\N_0$ we have on the one side 
  \begin{align*}
     \begin{pmatrix} z\psi(z,x_0) \\ -\psi^\qp(z,x_0) \end{pmatrix}  & =  \Gamma_0 M(z,t_0) M(z,t_n)^{-1} \Gamma_n^{-1} \begin{pmatrix} z\psi(z,x_n) \\ -\psi^\qp(z,x_n) \end{pmatrix}, \quad z\in\C,
  \end{align*}
  and on the other side (due to Lemma~\ref{lemEquCanSys}) also 
  \begin{align*}
   \begin{pmatrix} z\psi(z,x_0) \\ -\psi^\qp(z,x_0) \end{pmatrix}  = \Omega_+(z,x_0,x_n) \begin{pmatrix} z\psi(z,x_n) \\ -\psi^\qp(z,x_n) \end{pmatrix}, \quad z\in\C,
  \end{align*}
  where $\Omega_\pm(z,\ledot,x_n)$ denotes the matrix valued solution of the system~\eqref{eqnEquCanSys} such that $\Omega_\pm(z,x_n,x_n)$ is the identity matrix.
  In view of~\cite[Problem~100]{dB68}, % \cite[Theorem~III]{dB61a}
   this gives 
  \begin{align*}
   M(z,t_0)M(z,t_n)^{-1} = \Gamma_0^{-1} \Omega_+(z,x_0,x_n) \Gamma_n, \quad z\in\C.
  \end{align*}
  Differentiating with respect to $z$ and evaluating at zero, we obtain
  \begin{align*}
   \int_{t_n}^{t_0} H(s)ds & = \int_{x_0}^{x_n} \E^{-s} ds \begin{pmatrix} \varepsilon_0\E^{\frac{x_0}{2}}\varepsilon_n\E^{\frac{x_n}{2}} & \varepsilon_0\E^{\frac{x_0}{2}}  \\  \varepsilon_n\E^{\frac{x_n}{2}} & 1 \end{pmatrix} \\
    & \qquad + \int_{x_0}^{x_n} u'(s)-u(s)\, ds \begin{pmatrix} \varepsilon_0\E^{\frac{x_0}{2}} +\varepsilon_n\E^{\frac{x_n}{2}} & 1 \\ 1 & 0 \end{pmatrix} \\
    & \qquad + \int_{x_0}^{x_n} \E^s \left(u'(s)-u(s)\right)^2 ds \begin{pmatrix} 1 & 0 \\ 0 & 0 \end{pmatrix} + \int_{x_0}^{x_n} \E^s  d\dip(s) \begin{pmatrix} 1 & 0 \\ 0 & 0 \end{pmatrix}. 
  \end{align*}
  Since $H$ is integrable near zero, the top-right entry shows that the integrals 
  \begin{align*}
   \int_{x_0}^{x_n} u'(s) - u(s)\, ds
  \end{align*}
  are uniformly bounded for all $n\in\N_0$. 
  Moreover, from the bottom-left entry we now infer that the quantities $\varepsilon_n\E^{\frac{x_n}{2}}$ are uniformly bounded for all $n\in\N_0$ as well. 
  Finally, the top-left entry implies that the sum 
  \begin{align*}
   \int_{x_0}^{x_n} \E^s \left(u'(s) - u(s)\right)^2 ds +  \int_{x_0}^{x_n} \E^{s} d\dip(s)
  \end{align*}
  is uniformly bounded for all $n\in\N_0$, which concludes the proof since $x_n\rightarrow\infty$. 
\end{proof}
   
   For a discrete set $\sigma$ of nonzero reals such that the limit~\eqref{eqnIPdens} exists in $[0,\infty)$ and such that the entire function $W$ is well-defined by~\eqref{eqnIPW}, we define the isospectral set $\Iso{\sigma}$ as the set of all those pairs $(u,\mu)\in\De$ whose associated spectra coincide with the set $\sigma$.
  It is an immediate consequence of Theorem~\ref{thmIP} that the isospectral set $\Iso{\sigma}$ is in one-to-one correspondence with the set\footnote{To be precise, this has to be interpreted appropriately in the trivial case when $\sigma$ is empty.} 
 \begin{align}\label{eqnLambdasigma}
  \Lambda_\sigma = \left\lbrace \kappa\in\R^\sigma \,\left|\; \sum_{\lambda\in\sigma} \frac{1}{\lambda^2} \frac{\E^{|\kappa_\lambda|}}{|\lambda \dot{W}(\lambda)|}  < \infty \right.\right\rbrace,
 \end{align} 
 by means of the bijection given by  
 \begin{align}\label{eqnIsoSpecTrans}
  (u,\mu) \mapsto \left\lbrace \kappa_\lambda \right\rbrace_{\lambda\in\sigma}.
 \end{align}
% One can show that this mapping is a homeomorphism upon equipping $\Iso{\sigma}$ with the topology inherited from the product topology of the weak topology on $H^1(\R)$ and the weak$^\ast$ topology on the space of finite Borel measures on $\R$, and upon equipping the set $\Lambda_\sigma$ with the product topology inherited from $\R^\sigma$. 
   
 Of course, the solution of the inverse spectral problem can easily be formulated in terms of the right/left norming constants as well. 
 However, the picture (that is, the condition corresponding to~\eqref{eqnIPlogccCond} on the asymptotic behavior) looks less symmetric. 

\begin{corollary}\label{corIPNC}
  Let $\sigma$ be a discrete set of nonzero reals such that the limit~\eqref{eqnIPdens} exists in $[0,\infty)$ and such that the entire function $W$ is well-defined by~\eqref{eqnIPW}.
 Moreover, for each $\lambda\in\sigma$ let $\gamma_{\lambda,\pm}^2\in\R$ such that $\lambda\gamma_{\lambda,\pm}^2>0$ and the sums
 \begin{align}
   \sum_{\lambda\in\sigma} \frac{1}{\lambda^2} \frac{1}{\lambda\gamma_{\lambda,\pm}^{2}}, && \sum_{\lambda\in\sigma}  \frac{1}{\lambda^2} \frac{\gamma_{\lambda,\pm}^{2}}{\lambda \dot{W}(\lambda)^2},
 \end{align}
 are finite. 
 Then there is a unique pair $(u,\mu)\in\De$ such that the associated spectrum coincides with $\sigma$ and the right/left norming constants are $\gamma_{\lambda,\pm}^2$ for each $\lambda\in\sigma$. 
\end{corollary}

In order to avoid misunderstandings, let us point out that this corollary has to be read as two separate statements according to the plus-minus alternative, that is, we are only able to prescribe either the right or the left norming constants. 

\begin{remark}
 If the set $\sigma$ is finite, then all conditions in Theorem~\ref{thmIP} and Corollary~\ref{corIPNC} are trivially satisfied. 
 The solution of the inverse spectral problem in this case has a particular simple form and can be written down explicitly in terms of the spectral data; see~\cite[Section~4]{ConservMP}.  
 We will see next that these kinds of solutions can be used to approximate solutions in the general case in a certain way. 
\end{remark}

%%%%%%%%%%%%%%%%%%%%%%%%%%%%%%%%%%%
% \section{Continuity properties of the inverse spectral transform}
%%%%%%%%%%%%%%%%%%%%%%%%%%%%%%%%%%%

 We conclude this section by establishing a continuity property for the inverse spectral transform.   
 In order to state it, let $(u_k,\mu_k)$ belong to $\De$ for every $k\in\N$ and denote all corresponding quantities in an obvious way with an additional subscript. 
 
  \begin{proposition}\label{propCont}
  We have\footnote{Note that the following condition is equivalent to locally uniform convergence $m_{k,\pm}\rightarrow m_\pm$ of the corresponding Weyl--Titchmarsh functions given as in~\eqref{eqnmpmIntRep}.}
  % The conditions~\eqref{eqnContSD1} and~\eqref{eqnContSD2} imply that the Weyl--Titchmarsh functions converge pointwise and thus locally uniformly by compactness. 
  % For the converse, \eqref{eqnContSD1} follows upon taking the real part of $m_{k,\pm}(\I)$. Taking the imaginary part shows~\eqref{eqnContSD2} for $\chi=1$. It remains to show~\eqref{eqnContSD2} for $\chi\in C_0(\R)$. By compactness, we may assume that the corresponding measures in the Herglotz representation converge weak$^\ast$ to some Borel measure on $\R$. Comparing the Herglotz--Nevanlinna function one obtains this way with the limit $m_\pm$ on the other side, one may identify this measure and conclude~\eqref{eqnContSD2}. 
  \begin{align}\label{eqnContSD1}
  \sum_{\lambda\in\sigma_k} \frac{1}{\lambda(1+\lambda^2)} \frac{1}{\lambda\gamma_{k,\lambda,\pm}^2} & \rightarrow \sum_{\lambda\in\sigma} \frac{1}{\lambda(1+\lambda^2)} \frac{1}{\lambda\gamma_{\lambda,\pm}^2}, \\ \label{eqnContSD2}
  \sum_{\lambda\in\sigma_k} \frac{\chi(\lambda)}{1+\lambda^2} \frac{1}{\lambda\gamma_{k,\lambda,\pm}^2} & \rightarrow \sum_{\lambda\in\sigma}   \frac{\chi(\lambda)}{1+\lambda^2} \frac{1}{\lambda\gamma_{\lambda,\pm}^2}, 
 \end{align}
   as $k\rightarrow\infty$ for all functions $\chi\in C(\R)$ such that the limit of $\chi(\lambda)$ as $|\lambda|\rightarrow\infty$ exists and is finite if and only if the functions $u_k$ converge to $u$ pointwise and 
  \begin{align}\begin{split}\label{eqnContW}
    & \int_{\pm\infty}^x \E^{\mp s} \left( \int_{\pm\infty}^s \E^{\pm r} \left(u_k'(r)\mp u_k(r)\right)^2 dr + \int_{\pm\infty}^s \E^{\pm r}d\dip_k(r)\right) ds \\ 
     & \qquad\qquad \rightarrow \int_{\pm\infty}^x \E^{\mp s} \left( \int_{\pm\infty}^s \E^{\pm r} \left(u'(r)\mp u(r)\right)^2 dr + \int_{\pm\infty}^s \E^{\pm r}d\dip(r)\right) ds
  \end{split}\end{align}   
   as $k\rightarrow\infty$ for all $x\in\R$.   
   In this case, the functions $u_k$ converge to $u$ even locally uniformly and the Borel measures $\mu_k$ converge to $\mu$ in the sense of distributions. 
%   Furthermore, if both of the plus-minus alternatives hold true, then the functions $u_k$ converge to $u$ uniformly as well as weakly in $H^1(\R)$, the Borel measures $\mu_k$ converge to $\mu$ in the weak$^\ast$ topology and the sequence $\mu_k(\R)$ converges to $\mu(\R)$.
  \end{proposition}

 \begin{proof}
  It follows from \cite[Proposition~6.2]{IndefiniteString} that the first condition on convergence of~\eqref{eqnContSD1} and~\eqref{eqnContSD2} is equivalent to 
  \begin{align}
    \label{eqnContAr} \int_0^\xi \Ar_{k,\pm}(s)ds & \rightarrow \int_0^\xi \Ar_{\pm}(s)ds, \\
    \label{eqnContvs} \int_0^\xi \left(\int_0^s \Ar_{k,\pm}(r)^2 dr + \int_0^s d\beta_{k,\pm}\right) ds & \rightarrow \int_0^\xi \left(\int_0^s \Ar_{\pm}(r)^2 dr + \int_0^s d\beta_{\pm}\right) ds,
  \end{align}
  as $k\rightarrow\infty$ for all $\xi\in\R_\mp$, in which case the convergence is uniform as long as $\xi$ stays bounded. 
  In conjunction with~\eqref{eqnuitofa}, one sees that this implies that the sequence $u_k(x)$ converges to $u(x)$ for all $x\in\R$.
  Moreover, upon employing a substitution in~\eqref{eqnContvs}, we readily infer that~\eqref{eqnContW} holds as $k\rightarrow\infty$ for all $x\in\R$ as well. 
   Both of these convergences are uniform as long as $x$ stays away from $\mp\infty$. 
   In addition, for every smooth function $h$ on $\R$ that vanishes near $\mp\infty$ and such that $h'$ has compact support we have the identity % use substitution and integration by parts
  \begin{align*}
   \int_\R h\, d\mu_k & = \int_\R h(x)\alpha_{k,\pm}(x)^2 dx \pm 2 \int_\R h(x) u_k(x)u_k'(x)dx  + \int_\R h\, d\dip_k \\ 
                                & = \int_{\R_\mp} h_\pm(\xi) \Ar_{k,\pm}(\xi)^2 d\xi \mp \int_\R h'(x) u_k(x)^2 dx + \int_{\R_\mp} h_\pm\,  d\beta_{k,\pm}   \\
                                & = \int_{\R_\mp} h_\pm''(\xi) \int_0^\xi \left(\int_0^s \Ar_{k,\pm}(r)^2 dr + \int_0^s d\beta_{k,\pm}\right) ds \, d\xi \mp \int_\R h'(x)u_k(x)^2 dx,
  \end{align*}
  with $h_\pm$ given by~\eqref{eqnDefhpm}. 
  Taking the limit $k\rightarrow\infty$ on the right-hand side gives
  \begin{align}\label{eqnConvmuweak}
   \int_\R h\, d\mu_k \rightarrow \int_\R h\, d\mu, \qquad k\rightarrow\infty,
  \end{align} 
  which shows that the Borel measures $\mu_k$ converge to $\mu$ in the sense of distributions. 
  
 For the converse direction, we now assume that the functions $u_k$ converge to $u$ pointwise and that~\eqref{eqnContW} holds as $k\rightarrow\infty$ for every $x\in\R$. 
  As observed above, the latter condition is equivalent to the fact that~\eqref{eqnContvs} holds as $k\rightarrow\infty$ for all $\xi\in\R_\mp$. 
  Upon expressing the integrals in~\eqref{eqnContAr} in terms of $u_k$ and $u$ by using~\eqref{eqnDefa} as well as a substitution, 
  % that is, 
  %\begin{align*}
  % \int_0^\xi \Ar_\pm(s)ds = - u(\mp\ln(\mp\xi)) \pm \int_{\pm\infty}^{\mp\ln(\mp\xi)} u(s)ds
  %\end{align*}
  we infer (use Lebesgue's dominated convergence theorem and note that convergence of~\eqref{eqnContvs} implies a uniform bound on the monotone integrands which yields a sufficient estimate for the functions $u_k$) 
  % use the equation between~\eqref{eqnBoundThetaPhi} and~\eqref{eqnphipmasymest} to get the bound for u. 
  that~\eqref{eqnContAr} holds as $k\rightarrow\infty$ for all $\xi\in\R_\mp$, which establishes the equivalence in the claim. 
%  
%  Finally, let us suppose that both of the plus-minus alternatives hold true. 
%  Then the first part of the proof shows that the functions $u_k$ converge to $u$ uniformly and that the sequence $\mu_k(\R)$ converges to $\mu(\R)$. 
%  In particular, the sequence $\mu_k(\R)$  is bounded and thus the remaining claims follow in a straightforward manner. 
 \end{proof}

 We endow $\De$ with the unique first countable topology\footnote{For example, equip the space of Herglotz--Nevanlinna functions with the % metrizable 
 topology of locally uniform convergence and take the topology inherited from the injection $(u,\mu)\mapsto (m_{+},m_-)$.} such that the sequence $(u_k,\mu_k)$ converges to $(u,\mu)$ if and only if the functions $u_k$ converge to $u$ pointwise and both of the plus-minus alternatives in~\eqref{eqnContW} hold as $k\rightarrow\infty$ for all $x\in\R$. 

\begin{corollary}\label{corCont}
 If the sequence $(u_k,\mu_k)$ converges to $(u,\mu)$ in $\De$, then the functions $u_k$ converge to $u$ uniformly as well as weakly in $H^1(\R)$, the Borel measures $\mu_k$ converge to $\mu$ in the weak$^\ast$ topology\footnote{We regard the space of finite Borel measures on $\R$ as the dual of the Banach space $C_0(\R)$.} and the sequence $\mu_k(\R)$ converges to $\mu(\R)$.
\end{corollary}

\begin{proof}
  We have seen in the first part of the proof of Proposition~\ref{propCont} that the given assumption implies that the functions $u_k$ converge to $u$ uniformly.
  After choosing a suitable partition of unity (two functions are sufficient) we infer from~\eqref{eqnConvmuweak} that the sequence $\mu_k(\R)$ converges to $\mu(\R)$. 
  In particular, the sequence $\mu_k(\R)$  is bounded which allows to deduce the remaining claims in a straightforward manner. 
  % Norm boundedness plus convergence on a dense subspace imply weak$^\ast$ convergence. 
\end{proof}

%%%%%%%%%%%%%%%%%%%%%%%%%%%%%%%%
\section{The conservative Camassa--Holm flow}\label{secCCH}
%%%%%%%%%%%%%%%%%%%%%%%%%%%%%%%% 

 Let us define {\em the conservative Camassa--Holm flow} $\Phi$ on $\De$  as a mapping
 \begin{align}
   \Phi:  \De\times\R \rightarrow\De
 \end{align}
 in the following way:
 Given a pair $(u,\mu)\in\De$ with associated spectrum $\sigma$ and logarithmic coupling constants $\kappa_\lambda$ for every $\lambda\in\sigma$, as well as a $t\in\R$, the corresponding image $\Phi^t(u,\mu)$ under $\Phi$ is defined as the unique pair in $\De$ (guaranteed to exist by Theorem~\ref{thmIP}) for which the associated spectrum coincides with $\sigma$ and the logarithmic coupling constants are   
   \begin{align}\label{eqnKappaEvo}
    \kappa_\lambda + \frac{t}{2\lambda}, \quad  \lambda\in\sigma.
  \end{align}
% What the mapping $\Phi^t: \De\rightarrow\De$ means in terms of the spectral quantities: 
% \begin{align*}
%  \sigma & \mapsto \sigma \\
%  \kappa_\lambda & \mapsto \kappa_\lambda + \frac{t}{2\lambda} \\ 
%  c_\lambda & \mapsto c_\lambda \E^{\frac{t}{2\lambda}} \\ 
%  \gamma_{\lambda,\pm}^2 & \mapsto \gamma_{\lambda,\pm}^2 \E^{\mp \frac{t}{2\lambda}}
% \end{align*}
 Of course, this definition is motivated by the well-known simple time evolution of spectral data for classical solutions of the Camassa--Holm equation, which constitutes the essence of our method of solution and is due to the underlying completely integrable structure; see \cite[Section~6]{besasz98}.
 
  It is clear that the isospectral sets $\Iso{\sigma}$ are invariant under the conservative Camassa--Holm flow. 
  Moreover, the bijection in~\eqref{eqnIsoSpecTrans} takes the conservative Camassa--Holm flow on $\Iso{\sigma}$ to a simple linear flow on $\Lambda_\sigma$ whose solutions are explicitly given by~\eqref{eqnKappaEvo}.  
 In conjunction with these facts, the trace formulas in Proposition~\ref{propTrF} give rise to conserved quantities for the flow; cf.\ \cite{fisc99, le05}.
    
 \begin{proposition}
  The two functionals 
 \begin{align}\label{eqnCQ}
    (u,\mu) & \mapsto \int_\R u(x) dx, & (u,\mu) & \mapsto \int_\R d\mu, 
 \end{align}
 on $\De$ are invariant under the conservative Camassa--Holm flow.
 \end{proposition}
 
 Since the Wronskian $W$ is uniquely determined by the spectrum $\sigma$, it is invariant under the conservative Camassa--Holm flow. 
 Therefore, we see that the functional  
  \begin{align}
   (u,\mu) \mapsto \int_\R \rho(x)dx
  \end{align}
  on $\De$ is invariant as well (recall the notation from Remark~\ref{remETW}). 
  In particular, this shows that the property of $\rho$ vanishing almost everywhere (or equivalent, the Borel measure $\dip$ being singular with respect to the Lebesgue measure) is preserved. 
  This case corresponds to the subclass of global conservative solutions of the Camassa--Holm equation in \cite{brco07, hora07}, whereas the function $\rho$ is present for the general class of global conservative solutions of the two-component Camassa--Holm system in \cite{grhora12}.  

 \begin{proposition}\label{propFlowCont}
  The conservative Camassa--Holm flow $\Phi$ is continuous.  
 \end{proposition}

 \begin{proof}
  Let $t$, $t_k\in\R$ for every $k\in\N$ such that $t_k\to t$ as $k\to\infty$ and suppose that the sequence $(u_k,\mu_k)$ converges to $(u,\mu)$ in $\De$ so that~\eqref{eqnContSD2} holds as $k\to\infty$ for all functions $\chi\in C(\R)$ such that the limit of $\chi(\lambda)$ as $|\lambda|\rightarrow\infty$ exists and is finite. 
  In conjunction with Proposition~\ref{propTrF}, we infer from Corollary~\ref{corCont} that there is an $\varepsilon>0$ such that the intersections $(-\varepsilon,\varepsilon)\cap\sigma$ and $(-\varepsilon,\varepsilon)\cap\sigma_k$ are empty for every $k\in\N$.
 Given any continuous function $\tau$ on $\R\backslash\lbrace0\rbrace$ such that the limit of $\tau(\lambda)$ as $|\lambda|\rightarrow\infty$ exists and is finite, we choose a function $\chi_\pm\in C(\R)$ such that 
  \begin{align*}
  \chi_{\pm}(\lambda) =  \tau(\lambda) \E^{\pm\frac{t}{2\lambda}}, \quad |\lambda|\geq \varepsilon,  
 \end{align*}
 as well as constants $K$, $T\in\R$ such that $|\tau(\lambda)|\leq K$ for all $\lambda\in\R$ with $|\lambda|\geq\varepsilon$ and such that $|t_k|\leq T$ for all $k\in\N$. 
 We then may estimate 
     \begin{align*} 
  & \left|  \sum_{\lambda\in\sigma_k} \frac{\tau(\lambda)}{1+\lambda^2} \frac{1}{\lambda\gamma_{k,\lambda,\pm}^2 \E^{\mp\frac{t_k}{2\lambda}} }- \sum_{\lambda\in\sigma}   \frac{\tau(\lambda)}{1+\lambda^2} \frac{1}{\lambda\gamma_{\lambda,\pm}^2  \E^{\mp\frac{t}{2\lambda}}} \right| \\
%  &   \sum_{\lambda\in\sigma_k} \frac{|\tau(\lambda)|}{1+\lambda^2} \frac{1}{\lambda\gamma_{k,\lambda,\pm}^2} \left|\E^{\pm\frac{t_k}{2\lambda}} - \E^{\pm\frac{t}{2\lambda}} \right| + \left| \sum_{\lambda\in\sigma_k} \frac{\tau(\lambda)\E^{\pm\frac{t}{2\lambda}}}{1+\lambda^2} \frac{1}{\lambda\gamma_{k,\lambda,\pm}^2}   -  \sum_{\lambda\in\sigma}   \frac{\tau(\lambda)\E^{\pm\frac{t}{2\lambda}}}{1+\lambda^2} \frac{1}{\lambda\gamma_{\lambda,\pm}^2}  \right| \\
%    &  K |t_k-t| \sum_{\lambda\in\sigma_k} \frac{1}{1+\lambda^2} \frac{1}{\lambda\gamma_{k,\lambda,\pm}^2} \frac{1}{2|\lambda|} \E^{\frac{T}{2|\lambda|}} + \left| \sum_{\lambda\in\sigma_k} \frac{\tau(\lambda)\E^{\pm\frac{t}{2\lambda}}}{1+\lambda^2} \frac{1}{\lambda\gamma_{k,\lambda,\pm}^2}   -  \sum_{\lambda\in\sigma}   \frac{\tau(\lambda)\E^{\pm\frac{t}{2\lambda}}}{1+\lambda^2} \frac{1}{\lambda\gamma_{\lambda,\pm}^2}  \right| \\
   & \quad \leq   \frac{K}{\varepsilon} \E^{\frac{T}{\varepsilon}} |t_k-t| \sum_{\lambda\in\sigma_k} \frac{1}{1+\lambda^2} \frac{1}{\lambda\gamma_{k,\lambda,\pm}^2}  + \left|\sum_{\lambda\in\sigma_k} \frac{\chi_\pm(\lambda)}{1+\lambda^2} \frac{1}{\lambda\gamma_{k,\lambda,\pm}^2} - \sum_{\lambda\in\sigma} \frac{\chi_\pm(\lambda)}{1+\lambda^2} \frac{1}{\lambda\gamma_{\lambda,\pm}^2} \right|
 \end{align*}
 for every $k\in\N$.
 Since it follows readily from our assumptions that the right-hand side always converges to zero as $k\rightarrow\infty$, we infer that  
    \begin{align*}
  \sum_{\lambda\in\sigma_k} \frac{1}{\lambda(1+\lambda^2)} \frac{1}{\lambda\gamma_{k,\lambda,\pm}^2 \E^{\mp\frac{t_k}{2\lambda}}}  & \rightarrow \sum_{\lambda\in\sigma} \frac{1}{\lambda(1+\lambda^2)} \frac{1}{\lambda\gamma_{\lambda,\pm}^2 \E^{\mp\frac{t}{2\lambda}}}, \\ 
  \sum_{\lambda\in\sigma_k} \frac{\chi(\lambda)}{1+\lambda^2} \frac{1}{\lambda\gamma_{k,\lambda,\pm}^2 \E^{\mp\frac{t_k}{2\lambda}}} & \rightarrow \sum_{\lambda\in\sigma}   \frac{\chi(\lambda)}{1+\lambda^2} \frac{1}{\lambda\gamma_{\lambda,\pm}^2 \E^{\mp\frac{t}{2\lambda}}}, 
 \end{align*}
   as $k\rightarrow\infty$ for all functions $\chi\in C(\R)$ such that the limit of $\chi(\lambda)$ as $|\lambda|\rightarrow\infty$ exists and is finite.  
   In view of Proposition~\ref{propCont} and the definition of the flow $\Phi$, this implies that the corresponding images $\Phi^{t_k}(u_k,\mu_k)$ converge to $\Phi^t(u,\mu)$. 
 \end{proof}

 For the global conservative solutions of the two-component Camassa--Holm system in \cite{grhora12}, continuity results similar to Proposition~\ref{propFlowCont} were obtained in \cite[Theorem~5.2]{grhora12} and \cite[Theorem~6.7]{grhora14}. 
 However, the topologies used in \cite{grhora12, grhora14} are defined in a much more intricate and less explicit way; cf.\ \cite[Lemma~6.4 and Lemma~6.5]{grhora12}. 
  
 As the main result of this section, we are now going to show that the integral curve $t\mapsto\Phi^t(u_0,\mu_0)$ for any fixed initial data $(u_0,\mu_0)\in\De$ defines a weak solution of  the two-component Camassa--Holm system~\eqref{eqnmy2CH}. 
 To this end, let us denote the pair $\Phi^t(u_0,\mu_0)$ with $(u(\ledot,t),\mu(\ledot,t))$ for every $t\in\R$, so that $u$ can be regarded as a function on $\R\times\R$.
 Note that the integral curve $t\mapsto\Phi^t(u_0,\mu_0)$ is continuous by Proposition~\ref{propFlowCont}, which guarantees that the function $u$ is at least continuous. 
 
\begin{theorem}\label{thmWeakSol}
 The pair $(u,\mu)$ is a weak solution of the two-component Camassa--Holm system~\eqref{eqnmy2CH} in the sense that for every test function $\varphi\in C_\cc^\infty(\R^2)$  we have 
 \begin{align}
 & \int_\R \int_\R u(x,t) \varphi_t(x,t) + \left(\frac{u(x,t)^2}{2} + P(x,t) \right) \varphi_x(x,t) \,dx \,dt = 0, \\
 \begin{split} 
 &  \int_\R \int_\R \varphi_t(x,t) + u(x,t) \varphi_x(x,t) \,d\mu(x,t) \,dt  \\ 
 &   \qquad\qquad\qquad\qquad = 2\int_\R \int_\R u(x,t)\left(\frac{u(x,t)^2}{2} - P(x,t) \right) \varphi_x(x,t) \,dx \,dt,
 \end{split}
 \end{align}
 where the function $P$ is given by 
 \begin{align}
  P(x,t) =  \frac{1}{4} \int_\R \E^{-|x-s|} u(s,t)^2 ds +  \frac{1}{4} \int_\R \E^{-|x-s|} d\mu(s,t), \quad x,\, t\in\R.
 \end{align}
\end{theorem}

\begin{proof}
 Let us denote all the quantities corresponding to $(u_0,\mu_0)$ in an obvious way with an additional subscript. 
 If the spectrum $\sigma_0$ is a finite set, then it follows from \cite[Section~5]{ConservMP}  that the pair $(u,\mu)$ is a weak solution of the two-component Camassa--Holm system~\eqref{eqnmy2CH}.
 Otherwise, we define a pair $(u_{k,0},\mu_{k,0})\in\De$ for every $k\in\N$ in such a way that the associated spectrum coincides with $\sigma_0\cap[-k,k]$ and the right/left norming constants are $\gamma_{0,\lambda,\pm}^2$ for every $\lambda\in\sigma_0\cap[-k,k]$. 
 Since the intersection $\sigma_0\cap[-k,k]$ is a finite set, the corresponding pairs $(u_k,\mu_k)$ obtained from the integral curves $t\mapsto \Phi^t(u_{k,0},\mu_{k,0})$ satisfy   
 \begin{align}\label{eqnWSk1}
 & \int_\R \int_\R u_k(x,t) \varphi_t(x,t) + \left(\frac{u_k(x,t)^2}{2} + P_k(x,t) \right) \varphi_x(x,t) \,dx \,dt = 0, \\
 \begin{split} \label{eqnWSk2}
 &  \int_\R \int_\R \varphi_t(x,t) + u_k(x,t) \varphi_x(x,t) \,d\mu_k(x,t) \,dt  \\ 
 &   \qquad\qquad\qquad\qquad = 2\int_\R \int_\R u_k(x,t)\left(\frac{u_k(x,t)^2}{2} - P_k(x,t) \right) \varphi_x(x,t) \,dx \,dt,
 \end{split}
 \end{align}
 for every test function $\varphi\in C_\cc^\infty(\R^2)$, where the function $P_k$ is given by 
 \begin{align*}
  P_k(x,t) =  \frac{1}{4} \int_\R \E^{-|x-s|} u_k(s,t)^2 ds +  \frac{1}{4} \int_\R \E^{-|x-s|} d\mu_k(s,t), \quad x,\, t\in\R.
 \end{align*}
 Moreover, our definitions and Proposition~\ref{propTrF} guarantee the bounds
 \begin{align*}
  u_k(x,t)^2 \leq \mu_k(\R,t) \leq \frac{1}{2} \sum_{\lambda\in\sigma_0}\frac{1}{\lambda^2}, \quad x,\,t\in\R,~k\in\N,
 \end{align*}
 and from Proposition~\ref{propCont} we infer that the functions $u_k(\ledot,t)$ converge to $u(\ledot,t)$ locally uniformly and that the Borel measures $\mu_k(\ledot,t)$ converge to $\mu(\ledot,t)$ in the weak$^\ast$ topology (to conclude this, we also used that the sequence $\mu_k(\R,t)$ is bounded)  for every fixed $t\in\R$.
 In particular, this shows that the sequence of functions $P_k$ is bounded as well and converges to $P$ at least pointwise. 
 Finally, upon passing to the limit $k\rightarrow\infty$ in~\eqref{eqnWSk1} and~\eqref{eqnWSk2}, we infer that the pair $(u,\mu)$ is a weak solution of the two-component Camassa--Holm system~\eqref{eqnmy2CH}. 
\end{proof}

 When the Borel measure $\dip_0$ corresponding to the pair $(u_0,\mu_0)$ vanishes identically and the distribution $\omega_0$ is non-negative/non-positive, then this property is preserved by the conservative Camassa--Holm flow in view of Proposition~\ref{propDefinite}; cf.\ \cite[Corollary~3.3]{coes98}. 
 In this case, the function $u$ is a weak solution of the Camassa--Holm equation~\eqref{eqnCH} in the sense that for every test function $\varphi\in C_\cc^\infty(\R^2)$  we have 
 \begin{align}
  \int_\R \int_\R u(x,t) \varphi_t(x,t) + \left(\frac{u(x,t)^2}{2} + P(x,t) \right) \varphi_x(x,t) \,dx \,dt = 0,
 \end{align}
 where the function $P$ is given by 
 \begin{align}
  P(x,t) =  \frac{1}{2} \int_\R \E^{-|x-s|} u(s,t)^2 ds + \frac{1}{4} \int_\R \E^{-|x-s|} u_x(s,t)^2 ds, \quad x,\, t\in\R. 
 \end{align}
 Under these assumptions, global weak solutions in this sense have been obtained before in \cite{coes98c, como00}. % \cite[Theorem~3.2]{coes98c}
 In the more challenging general case, existence of global weak solutions as in Theorem~\ref{thmWeakSol} has been established in \cite{brco07, hora07, grhora12} by means of an elaborate transformation to Lagrangian coordinates. 

 \begin{remark}
  If the spectrum $\sigma_0$ associated with the pair $(u_0,\mu_0)$ is a finite set, then we recover the special class of conservative multi-peakon solutions \cite{hora07a}, which can be written down explicitly in terms of the spectral data \cite{ConservMP}. 
  We have seen in the proof of Theorem~\ref{thmWeakSol} that in the general case, our weak solution $(u,\mu)$ of the two-component Camassa--Holm system~\eqref{eqnmy2CH} can be approximated by a sequence of conservative multi-peakon solutions in a certain way; cf.\ \cite{hora06, hora08}.
  \end{remark}
   
 We conclude this section with a comment on the long-time behavior of the conservative Camassa--Holm flow: 
 Upon employing the method introduced in~\cite{CouplingProblem}, it is possible to show that our weak solution $(u,\mu)$ of the two-component Camassa--Holm system~\eqref{eqnmy2CH} asymptotically splits into a (in general infinite) train of single peakons, each corresponding to an eigenvalue $\lambda\in\sigma_0$ of the underlying spectral problem.  
 For classical solutions of the Camassa--Holm equation~\eqref{eqnCH}, this behavior was anticipated by McKean~\cite{mc03} (in accordance with numerical observations in~\cite{cahohy94}) and first proved in \cite{IsospecCH} (see also \cite{besasz00} and \cite{ConservMP} for the multi-peakon case and \cite{li09} for a particular class of low-regularity solutions).

\appendix

%%%%%%%%%%%%%%%%%%%%%%%%%%%%%%%%
\section{The basic differential equation}\label{appRelCan}
%%%%%%%%%%%%%%%%%%%%%%%%%%%%%%%%

 Throughout this appendix, let $u$ be a real-valued function in $H^{1}_{\loc}(\R)$ and $\dip$ be a non-negative Borel measure on $\R$.  
 We define the distribution $\omega$ in $H^{-1}_\loc(\R)$ by\footnote{The space $H^1_\cc(\R)$ consists of all those functions in $H^1(\R)$ which have compact support.} 
\begin{align}\label{eqnDefomegaapp}
 \omega(h) = \int_\R u(x)h(x)dx + \int_\R u'(x)h'(x)dx, \quad h\in H^1_\cc(\R), 
\end{align}
 so that $\omega=u-u''$ formally, and consider the ordinary differential equation 
 \begin{align}\label{eqnDEho}
  -f'' + \frac{1}{4} f = z\, \omega f + z^2 \dip f, 
 \end{align}
 where $z$ is a complex spectral parameter. 
 Of course, this differential equation has to be understood in a distributional sense in general; cf.\ \cite{IndefiniteString, gewe14, sash03}.   
  
  \begin{definition}\label{defSolution}
  A solution of~\eqref{eqnDEho} is a function $f\in H^1_{\loc}(\R)$ such that 
 \begin{align}\label{eqnDEweakform}
   \int_{\R} f'(x) h'(x) dx + \frac{1}{4} \int_\R f(x)h(x)dx = z\, \omega(fh) + z^2 \int_\R f h \,d\dip 
 \end{align} 
 for every function $h\in H^1_\cc(\R)$.
 \end{definition}
 
% It is readily verified that this definition is consistent with the notion of solution used in \cite{be89, bebrwe08, bebrwe12, LeftDefiniteSL, ConservMP, CHPencil, IsospecCH, MeasureSL} when $\omega$ is a real-valued Borel measure on $\R$, that is, when $u'$ admits a representative that is locally of bounded variation. 
 
 In order to make this notion of solution more accessible, we will first show that the differential equation~\eqref{eqnDEho} is equivalent to the first order system  
 \begin{align}\begin{split}\label{eqnEquCanSys}
 \begin{pmatrix} 0 & 1 \\ -1 & 0 \end{pmatrix}F_\pm' & = \mp\frac{1}{2} \begin{pmatrix} 0 & 1 \\ 1 & 0 \end{pmatrix} F_\pm + z \begin{pmatrix} \alpha_\pm^2  & \alpha_\pm \\ \alpha_\pm & 1 \end{pmatrix} F_\pm + z\begin{pmatrix} \dip & 0 \\ 0 & 0 \end{pmatrix} F_\pm,
\end{split}\end{align} 
 where we introduced $\alpha_\pm=- u'\pm u$. 
Since $\dip$ is allowed to be a genuine Borel measure, this system has to be understood as a measure differential equation \cite{at64, be89, MeasureSL, pe88} in general. 
Solutions of such an equation are not necessarily continuous but only locally of bounded variation and, in order to guarantee unique solvability of initial value problems, one has to impose a suitable normalization. 
Here we will require solutions to be left-continuous, which is implicitly contained in the following definition: 
A solution of the system~\eqref{eqnEquCanSys} is a function $F_\pm:\R\rightarrow\C^2$ with locally bounded variation such that 
\begin{align}\begin{split}\label{eqnEquCanSysInt}
   \left.F_\pm\right|_{x}^{y} =  \mp \frac{1}{2} \int_x^y \begin{pmatrix}  -1  & 0 \\ 0 & 1 \end{pmatrix}  F_\pm(s) ds & + z \int_x^y \begin{pmatrix}-\alpha_\pm(s)  & -1 \\ \alpha_\pm(s)^2 & \alpha_\pm(s) \end{pmatrix} F_\pm(s)ds \\
        & + z \int_{x}^y \begin{pmatrix} 0  & 0 \\ 1 & 0 \end{pmatrix} F_\pm \,d\dip   
\end{split}\end{align}
for all $x$, $y\in\R$. In this case, the first component of $F_\pm$ is clearly locally absolutely continuous and the second component is left-continuous; cf.\ \eqref{eqnDefintmu}. 
With this notion of solutions, initial value problems for the system~\eqref{eqnEquCanSys} are always uniquely solvable; see, for example, \cite[Section~11.8]{at64}, \cite[Theorem~1.1]{be89}, \cite[Theorem~A.2]{MeasureSL}.

\begin{lemma}\label{lemEquCanSys}
 If the function $f$ is a solution of the differential equation~\eqref{eqnDEho}, then there is a unique left-continuous function $f^\qd$ such that 
 \begin{align}\label{eqnfqpm} 
     f^\qd(x) = f'(x) - z u'(x) f(x) 
\end{align} 
 for almost all $x\in\R$ and the function 
 \begin{align}\label{eqnEquCanSysVec}
  \begin{pmatrix} z f \\ -f^\qd \pm\left(\frac{1}{2} - z u\right)f \end{pmatrix}  
 \end{align}
 is a solution of the system~\eqref{eqnEquCanSys}. 
 Conversely, if the function $F_\pm$ is a solution of the system~\eqref{eqnEquCanSys}, then its first component is a solution of the differential equation~\eqref{eqnDEho}. 
\end{lemma}

\begin{proof}
  First one notes that for every $c\in\R$ the equation~\eqref{eqnDEweakform} takes the form
  \begin{align*}
   \int_\R h'(x) \left( f'(x) - z u'(x)f(x) - \frac{1}{4} \int_c^x f(s)ds\right. & + z\int_c^x u(s)f(s)+u'(s)f'(s)\,ds  \\ & + \left. z^2 \int_c^x f \,d\dip \right) dx = 0
  \end{align*}
 upon integrating by parts. 
 Now from this we infer that a function $f\in H^1_\loc(\R)$ is a solution of~\eqref{eqnDEho} if and only if there is a $c\in\R$ and a constant $d\in\C$ such that 
 \begin{align}\begin{split}\label{eqnDEweakderiv}
  f'(x) - z u'(x)f(x) = d +  \frac{1}{4} \int_c^x f(s)ds & -z\int_c^x u(s)f(s)+u'(s)f'(s)\,ds \\ &  - z^2 \int_c^x f \,d\dip
 \end{split}\end{align}
 for almost all $x\in\R$. 
 So if $f$ is a solution of~\eqref{eqnDEho}, then this guarantees that there is a unique left-continuous function $f^\qd$ such that~\eqref{eqnfqpm} holds for almost all $x\in\R$. 
 Let us denote the second component in~\eqref{eqnEquCanSysVec} with  $-f^\qpm$ so that  
 \begin{align*}
  f^\qpm(x) = f'(x) \mp \frac{1}{2} f(x) + z\alpha_\pm(x) f(x)
 \end{align*}  
 for almost all $x\in\R$. 
 Then it is straightforward to show that 
 \begin{align*}
  \pm\frac{1}{2} & \int_x^y  f^\qpm(s)ds + z\int_x^y \alpha_\pm(s)^2 zf(s) - \alpha_\pm(s) f^\qpm(s)\,ds \\
    & \qquad = \left.\pm \biggl(\frac{1}{2}  - z u \biggr)f\right|_{x}^y - \frac{1}{4} \int_x^y f(s)ds + z \int_x^y u(s)f(s) + u'(s)f'(s)\,ds
 \end{align*}
 for all $x$, $y\in\R$. 
 In combination with~\eqref{eqnDEweakderiv}, we end up with  
  \begin{align*}
  \left. -f^\qpm \right|_{x}^{y} = \pm\frac{1}{2} \int_x^y f^\qpm(s)ds & + z\int_x^y \alpha_\pm(s)^2 zf(s) - \alpha_\pm(s) f^\qpm(s)\,ds  + z \int_x^y zf \,d\dip
  \end{align*}
 for all $x$, $y\in\R$, which shows that~\eqref{eqnEquCanSysVec} is a solution of the system~\eqref{eqnEquCanSys}. 
 
 Now suppose that $F_\pm$ is a solution of the system~\eqref{eqnEquCanSys} and denote the respective components with subscripts.  
 The first component of~\eqref{eqnEquCanSysInt} shows that $F_{\pm,1}$ belongs to $H^1_\loc(\R)$ with 
 \begin{align}\label{eqnDEweakint1comp}
  z F_{\pm,2}(x) = -F_{\pm,1}'(x)  \pm \frac{1}{2}F_{\pm,1}(x) - z \alpha_\pm(x) F_{\pm,1}(x) 
 \end{align}
 for almost all $x\in\R$. In combination with the second component of~\eqref{eqnEquCanSysInt} this gives 
 \begin{align*}
  F_{\pm,1}'(x) & =  \pm\frac{1}{2} F_{\pm,1}(x) -  z\alpha_\pm(x) F_{\pm,1}(x) - z F_{\pm,2}(c) \pm\frac{1}{2} \int_c^x z F_{\pm,2}(s)ds \\ & \qquad - z^2\int_c^x \alpha_\pm(s)^2 F_{\pm,1}(s)ds  -z \int_c^x \alpha_\pm(s)\, zF_{\pm,2}(s)ds - z^2 \int_c^x F_{\pm,1} \,d\dip 
 \end{align*}
 for some $c\in\R$ and almost all $x\in\R$.  
 Plugging~\eqref{eqnDEweakint1comp} twice into this equation, we see that~\eqref{eqnDEweakderiv} holds with $f$ replaced by $F_{\pm,1}$ for some constant $d$.
 As noted before, this guarantees that $F_{\pm,1}$ is a solution of the differential equation~\eqref{eqnDEho}.  
\end{proof}
  
 The auxiliary function $f^\qd$ introduced in Lemma~\ref{lemEquCanSys} allows us to state the following basic existence and uniqueness result for the differential equation~\eqref{eqnDEho}. 
  
\begin{corollary}\label{corEE2}
 For every $c\in\R$ and $d_1$, $d_2\in\C$ there is a unique solution $f$ of the differential equation~\eqref{eqnDEho} with the initial conditions
 \begin{align}\label{eqnDEiv}
 f(c) = d_1 \qquad\text{and}\qquad  f^\qd(c) = d_2.
 \end{align} 
 If $d_1$, $d_2$ and $z$ are real, then the solution $f$ is real-valued as well. 
\end{corollary}
  
  \begin{proof}
   If $z$ is zero, then the solutions of the differential equation~\eqref{eqnDEho} are given explicitly and the claim is obvious. 
   Otherwise, it follows from unique solvability of initial value problems for the system~\eqref{eqnEquCanSys} in conjunction with Lemma~\ref{lemEquCanSys}. 
  \end{proof}
    
 We are now left to introduce the Wronskian $W(f,g)$ of two solutions $f$, $g$ of the differential equation~\eqref{eqnDEho}.
 Although the following result is essentially a consequence of Lemma~\ref{lemEquCanSys} and the fact that the traces of the matrices in~\eqref{eqnEquCanSysInt} vanish, we provide an independent direct proof.  

\begin{corollary}\label{corWronskian}
 For any two solutions $f$, $g$ of the differential equation~\eqref{eqnDEho} there is a unique complex number $W(f,g)$ such that 
\begin{align}\label{eqnDefWron}
W(f,g) = f(x)g'(x) -f'(x)g(x)
\end{align} 
 for almost all $x\in\R$. This number is zero if and only if the functions $f$ and $g$ are linearly dependent.
\end{corollary}

\begin{proof}
 First of all, note that if $f$ is a solution of the differential equation~\eqref{eqnDEho} and $h\in H^1_\loc(\R)$, then~\eqref{eqnDEweakderiv} and the integration by parts formula~\eqref{eqnPI} show that 
  \begin{align}\begin{split}\label{eqnSIP}
    \left. f^\qd h\right|_{x}^y & = \int_x^y f'(s)h'(s)ds + \frac{1}{4} \int_x^y f(s)h(s)ds \\ & \qquad - z\int_x^y  u(s)f(s)h(s) + u'(s)(fh)'(s)\,ds - z^2 \int_x^y fh\, d\dip
  \end{split}\end{align}
 for all $x$, $y\in\R$. 
%   Alternatively  
% \begin{align*}
%  \left. f\,g^\qpm\right|_{x}^y =  \int_x^y  f^\qpm(s) g^\qpm(s)ds - \int_x^y \alpha_\pm(s)^2\, zf(s)\,zg(s)\,ds - \int_x^y zf\, zg\, d\dip
% \end{align*} 
 If $g$ is another solution of the differential equation~\eqref{eqnDEho}, then 
 \begin{align}\label{eqnWronskiAE}
  f(x)g'(x) - f'(x) g(x) =  f(x)g^\qd(x) - f^\qd(x) g(x)
 \end{align}
 for almost all $x\in\R$. 
 Employing~\eqref{eqnSIP}, one sees that the right-hand side of~\eqref{eqnWronskiAE} is constant, which proves existence of the number $W(f,g)$. 
 The remaining claim follows in a standard way upon utilizing the uniqueness part in Corollary~\ref{corEE2}. 
% Clearly, the number $W(f,g)$ is zero if $f$ and $g$ are linearly dependent. 
% Conversely, if $W(f,g)$ is zero, then (we may assume that $g$ does not vanish identically) there is a constant $K\in\C$ such that $f(c) = K g(c)$ and $f^\qd(c) = K g^\qd(c)$ for some $c\in\R$. 
% Now by the uniqueness part of Corollary~\ref{corEE2} and linearity of~\eqref{eqnDEho}, we infer that $f$ and $g$ are linearly dependent. 
\end{proof}

Let us mention that one can easily get rid of the constant potential term in the first order system~\eqref{eqnEquCanSys} by scaling solutions with the matrix function
\begin{align}
 \begin{pmatrix} \E^{\mp\frac{x}{2}} & 0 \\ 0 & \E^{\pm\frac{x}{2}} \end{pmatrix}, \quad x\in\R, 
\end{align}
at the cost of rescaling the remaining matrix coefficients in~\eqref{eqnEquCanSys} as well. 
% That is, if $F_\pm$ is a solution of~\eqref{eqnEquCanSys}, then the function 
%\begin{align}
% G_\pm(x) = \begin{pmatrix} \E^{\mp\frac{x}{2}} & 0 \\ 0 & \E^{\pm\frac{x}{2}} \end{pmatrix} F_\pm(x), \quad x\in\R, 
%\end{align}
% is a solution of the first order system in integral form 
%\begin{align}\begin{split}
%   \left.G_\pm\right|_{x}^y & = z \int_x^y \begin{pmatrix}-\alpha_\pm(s)  & -\E^{\mp s} \\ \alpha_\pm(s)^2 \E^{\pm s} & \alpha_\pm(s) \end{pmatrix} G_\pm(s)ds 
%        + z \int_{x}^y \begin{pmatrix} 0  & 0 \\ \E^{\pm s} & 0 \end{pmatrix} G_\pm(s) d\dip(s).   
%\end{split}\end{align}
 It is even possible to further transform~\eqref{eqnEquCanSys} into a canonical system in standard form
\begin{align}
  \begin{pmatrix} 0 & 1 \\ -1 & 0 \end{pmatrix} G_\pm' = z H_\pm G_\pm,
\end{align}
 with a locally integrable, non-negative and trace normed matrix function $H_\pm$. 
 Since this is of no need for our purposes, we will not do this here but refer to \cite[Proof of Theorem~6.1]{IndefiniteString}, where such a transformation has been employed.  
 
% \bigskip
%\noindent
%{\bf Acknowledgments.}

\end{document}